\documentclass[final]{siamltex}
\usepackage{amssymb,amsmath}
\usepackage{graphics} 
\usepackage{hyperref}
\usepackage{xcolor}
\usepackage{url}

\usepackage{caption}
\usepackage{graphicx}
\usepackage{subfig}
\usepackage{floatrow}
\newfloatcommand{capbtabbox}{table}[][\FBwidth]
\usepackage{changepage}

\usepackage[numbers]{natbib}
\date{first: August 14, 2018; then: August 20, 2019; subm: Dec 26, 2020; latest: August 17, 2021}

\title{\large{\bf Large Deviations for Additive Functionals\\ of Reflected Jump-Diffusions}}
                                                       
\author{Lea Popovic\thanks{Department of 
Mathematics and Statistics, Concordia University, Montreal, Canada (\tt lea.popovic@concordia.ca)}\and Giovanni Zoroddu 
\thanks{Department of 
Mathematics and Statistics, Concordia University, Montreal, Canada   (\tt giovanni.zoroddu@concordia.ca)}}

\begin{document}
 \maketitle

\newtheorem{remark}[theorem]{Remark}

\def \non{{\nonumber}}
\def \hat{\widehat}
\def \tilde{\widetilde}
\def \bar{\overline}
\def\ep{\epsilon}
\def\N{{\mathbb{N}}}
\def\R{{\mathbb{R}}}
\def\Z{{\mathbb{Z}}}
\def\p{\partial}
\def\ol{\overline}
\def\l{\left}
\def\r{\right}
\def\p{\partial}

\vspace{.5in}

\begin{abstract}

We consider a jump-diffusion process  on a bounded domain with reflection at the boundary, and establish long-term results for a general additive process of its path. This includes the long-term behaviour of its occupation and local times in the interior and on the boundaries. We derive a characterization of the large deviation rate function which quantifies the rate of exponential decay of probabilities of rare events for these additive processes. The characterization relies on a solution of a partial integro-differential equation (PIDE) with boundary constraints. We develop a practical implementation of our results in terms of a numerical solution for the PIDE. We illustrate the method on a few standard examples (reflected Brownian motion, birth-death processes) and on a particular reflected jump-diffusion model arising from applications to biochemical reactions. 

\end{abstract}

\pagestyle{myheadings} \thispagestyle{plain}

\vspace{.2in}

\vspace{.2in}

\section{Introduction}

Many applied stochastic systems in finance, economics, queueing theory, and electrical engineering are modeled by jump-diffusions (\cite{K07, R03, T03, WG02, WG03, DM15}). Some important properties of climate systems were explained by the addition of jumps in modeling (\cite{D99, DD09}). In many biological models, the inclusion of jumps in diffusion models has been useful: in neuronal systems (\cite{JBHD11, SG13}), as well as in ecology and evolution (\cite{JMW13, CFM06}). Our motivation comes from problems in systems biology, where basic intracellular processes are modeled by chemical reaction networks (\cite{BKPR06, W09, W18, GAK15}). Due to the complexity of multi-scale features in chemical reaction systems, the most appropriate approximation of their inherent stochasticity may require jump-diffusion models.

Although many practical results for L\'evy processes are well explored, relatively fewer are available for the more general jump-diffusions. In some applications, modeling by pure L\'evy processes is inadequate as both jump and diffusion rates will genuinely depend on the current state of the system. For example, in chemical reaction networks, jump rates and diffusion coefficients are derived from rates of interactions between different molecular species, and these rates inherently depend on the amount of species types  presently in the system. Consequently, one needs to consider stochastic differential equations driven by Poisson random measures.
Many systems also take values on positive and bounded spaces, because of the natural constraints on the amounts of species in the system. For chemical reaction models, the counts of molecules often satisfy some conservation relations in the system which keep these counts bounded above. The same is true for ecological constraints based on carrying capacities, and for some financial and engineering models with restrictions.
The reflection of the process when it reaches the boundary of its domain may be built into model dynamics. For example, jump-diffusions modeling chemical reaction systems need to have oblique reflections at the boundaries defined in order to match the same behaviour of jump Markov models (\cite{AHLW19}). 

Long term behaviour of these models reveals their stability and equilibria, and the portion of time spent in different parts of the state space. Ergodic theory quantifies averages of integrated functions of process paths, martingale methods provide standard deviations from these averages, and   large deviation theory provides more detailed results on rare departures from average behaviour. Large deviation rate functions quantify the values of dominant terms in integrated exponential functions of the process. Our study is motivated by the fact that many important biological mechanisms rely on the occurrence of rare events. In some mechanisms they lead to transitions to a new stable state and these transitions typically arise from the intrinsic stochasticity of the system (\cite{BvOC11,BQ10}). Due to population proliferation (cell growth+division, or species demography) rare events have many opportunities to occur and, though rare on the level of an individual molecule, occur with reasonable probability on the scale of the whole population.
Time additive functionals of the process, or dynamical observables, are of particular interest for experimental studies with limited access to precise values at specific time points, and easier access to empirical distributions. In chemical reaction network models occupation measures can be used to distinguish the orders of magnitudes of certain subsets of reactions within the system, and thus help determine which approximating  model is most appropriate (\cite{McSP14}).

We are interested in computing long time statistics for time additive functionals of reflected jump-diffusions. Exact and explicit expressions can be found only in very special cases, and for most processes of interest one has to rely on numerical methods of evaluation.
While long-term averages and standard deviations can be simulated using Monte Carlo methods (\cite{AG07}), large deviations are non-trivial to assess numerically. For Markov processes and small noise diffusions one can devise simulation methods of  rare events using large deviation rate functions and importance sampling techniques (\cite{B04, vEW12}), the efficiency of which is  dependent on the application in question  (\cite{RT09}, e.g. in climate modeling \cite{RWB18})). For chemical reaction dynamics there are several numerical methods for simulating functional large deviations of complete paths in small noise diffusion models (e.g. \cite{EWvE04,  vEW12}), or in pure jump Markov processes (e.g. \cite{AWtW05}). However, large deviations of time-integrated additive functionals should be less arduous than functional  large deviation paths. The framework for additive functionals relies on taking a limit as the time of integration approaches infinity, rather than a limit in which the noise of the model vanishes.
 
The first theoretical results for large deviations of occupation times and empirical measures for ergodic Markov processes date back to Donsker-Varadhan (\cite{DVIII76,DVIV83}), with additional approaches established by G\"artner (\cite{G77}) and Stroock (\cite{S84}). In this paper we use the results of Fleming-Sheu-Soner (\cite{FSS87}) to get the large deviation principle for additive functionals of reflected jump-diffusion processes assuming they are ergodic. This technique ensures the existence and uniqueness of a solution to a boundary value partial integro-differential equation (PIDE), which identifies the limiting logarithmic moment generating function of the additive process (for fixed parameter value in the generating function)  as an eigenvalue for a second order linear operator paired with its eigenfunction. Only in some special cases  is the explicit form for this eigenvalue available (c.f. \cite{FKZ15} for reflected one-dimensional Brownian motion and its  local time on the boundary).  

For one-dimensional reflected diffusions and L\'evy processes, a similar boundary value PIDE was obtained to characterize the limiting logarithmic moment generating function for additive functionals in (\cite{GW15}, \cite{AAGP15} Sec 14.4). Our results cover the more general reflected jump-diffusions  in multi-dimensional space, and provide sufficient assumptions for the existence and uniqueness of a solution to this PIDE, in order to establish the large deviation principle for the process. We thus achieve more general theoretical conclusions with the potential of a greater range of applicability. 

In order to use our result in practice, we additionally provide a numerical technique for calculating the limiting logarithmic moment generating function based on numerically solving the eigenvalue problem associated to the PIDEs. We do this by way of finite-differences to approximate the derivatives and numerical quadrature or weighted sums for the integral term. Similar methods can be implemented for multidimensional problems with a small number of variables, with some care regarding an efficient evaluation of the integral term. We test our results and their numerical implementation on two special cases, a reflected Brownian motion and a reflected  birth-death process, for which a comparison with analytic solutions is possible. We then use our results on an application that is analytically intractable: an example of a (jump)-diffusion model that approximates a system of chemical reactions. We use our large deviation results to calculate the long term mean local time at its two deterministic stable states, and the probability of departures from it. We then use this additive functional to compare the long-term behaviour of two types of approximate models for this system: a reflected diffusion  process (based on the Constrained Langevin approximation developed in \cite{LW19}), and a reflected jump-diffusion process that allows a subset of its dynamics to have sizeable noise. 


Our paper is structured as follows. Section 2.1 introduces the reflected jump-diffusion model, the  associated boundary process and the additive functional; Section 2.2 presents the results for the large deviations of this additive process; while Section 2.3 specializes the results to the case of one dimensional reflected jump-diffusions. Section 3 presents a numerical scheme used to implement the calculation of the large deviation rate function from the associated PIDE with boundary conditions in one dimension; in Sections 3.1 and 3.2 we derive analytical expressions for the rate function in two special case examples, and compare it to its numerical approximation; finally, in Section 3.3 we give numerical results for our application of a system of chemical reactions given by  a jump Markov model and compare it to a diffusion model, as well as to a jump-diffusion model. A description of the numerical scheme is presented in the Section 4 appendix.

\section{Large deviation result} 

\subsection{Preliminaries and assumptions} Let $(\Omega,\mathcal{F},\{\mathcal{F}(t)\}_{t\geq0},\mathbb{P})$ be a filtered probability space and let $m\geq0$ and $d\geq1$. Denote the Skorokhod space of right-continuous functions with finite left-hand limits (RCLL) by $\mathcal{D}\equiv\mathcal{D}_{[0,\infty)}(\mathbb{R}^d)$ and the set of RCLL functions in $\bar S$ by $\mathcal{D}(S)=\{Z\in \mathcal{D}: Z(t)\in \bar{S}, t\geq0\}$, where $S\subset\mathbb{R}^d$ will be made precise in the context of normal or oblique reflections later in this section. Consider a $d$-dimensional jump diffusion $X=\{X(t)\}_{t\geq0}$ to be a $\mathcal{D}$-measurable process that satisfies
\begin{equation}\label{eq-X}
X(t)=X(0)+\int_{0}^{t}\mu(X(s))ds+\int_{0}^{t}\sigma(X(s))dB(s)+\int_{0}^{t}\!\!\int_{\mathcal{M}}\gamma(X(s-),y)N(ds,dy),\quad X(0)=x_0\in \mathbb{R}^d,
\end{equation}
where $B=\{B(t)\}_{t\geq0}$ is an $m$-dimensional Brownian motion, adapted to, and a martingale with respect to $\{\mathcal{F}(t)\}_{t\geq0}$; $\mathcal{M}$ is a Borel measurable subset of $\mathbb{R}^n$; and the random counting measure $N(t,\mathcal{M})=\sum_{0<s\leq t}\textbf{1}_{\{\gamma(X(s-),y)\neq 0: y\in\mathcal{M}\}}$ is adapted to $\{\mathcal{F}(t)\}_{t\geq0}$, independent of $B$, and has state-dependent intensity measure $t\cdot\nu_z({A})=\mathbb{E}\left[N(t,{A})|X(t-)=z\right]$ for each $z\in\mathbb{R}^d$, ${A}\subseteq\mathcal{M}$ such that \mbox{$\int_{|y|\leq 1} |y| \sup_{z\in\mathbb{R}^d} \nu_z(dy)<\infty$}. Since the jumps of (\ref{eq-X}) are assumed to be of finite variation on finite time intervals, the sum of all jumps is well defined and we may rewrite the jump integral term in its canonical form $\int_{0}^{t}\int_{\mathcal{M}}\gamma(X(s-),y)N(ds,dy)=\sum_{0<s\leq t:\Delta X(s)\neq 0}\Delta X(s)$ where $\Delta X(s) = X(s)-X(s-)=\gamma(X(s-),y)$ and $y$ is sampled at a rate and magnitude induced by the measure $\nu_z$ (\cite{LS12}). Note that, after compensating for the mean, we have that for each $z\in\mathbb{R}^d, \mathcal{A}\subseteq\mathcal{M}, \tilde{N}(t,\mathcal{A})=N(t,\mathcal{A})-t\cdot\nu_z(\mathcal{A})$ is a martingale-valued measure.

We assume that the measurable functions $\mu:\mathbb{R}^d\to\mathbb{R}^d, \sigma:\mathbb{R}^d\to\mathbb{R}^{d\times m},$ and $\gamma:\mathbb{R}^d\times\mathcal{M}\to \mathbb{R}^{d}$ are Lipschitz in order to ensure the existence of a unique strong solution to the stochastic differential equation (SDE)  (\ref{eq-X}) (c.f. Theorem V.7, \cite{P13}). We additionally impose a linear growth condition on the coefficients, that is, we assume $\exists C_1>0$ such that for all $x_1,x_2\in\mathbb{R}^d$
\begin{equation}\label{cond-lipschitz}
||\mu(x_1)-\mu(x_2)||^2+||\sigma(x_1)-\sigma(x_2)||^2+||\int_{\mathcal{M}}\gamma(x_1,y)\nu_{x_1}(dy)-\int_{\mathcal{M}}\gamma(x_2,y)\nu_{x_2}(dy)||^2\leq C_1||x_1-x_2||^2,
\end{equation}
and $\exists C_2>0$ such that for all $x\in\mathbb{R}^d$
\begin{equation}\label{cond-linear_growth}
||\mu(x)||^2+||\sigma(x)||^2+\int_{\mathcal{M}}||\gamma(x,y)||^2\nu_x(dy)\leq C_2(1+|x|^2),
\end{equation}
where $||\cdot||$ denotes the appropriate Euclidean norm.\\

We next introduce a solution to the stochastic differential equation with reflection (SDER). 
It is related to the {\bf Skorokhod problem (SP)}, which for a given process $\phi$ defines an {\bf associated boundary process} $\eta$,  with finite variation on finite time intervals, whose variation increases only when $\phi$ is on the boundary of a given domain, in such a way that ensures the reflected process $\varphi=\phi+\eta$ remains in the domain. The SP has been established for various processes in different domains including: multidimensional diffusions in convex domains (\cite{T79}),  in general domains satisfying conditions (A), (B) (defined below) and admissibility conditions (\cite{LS84}), where the latter conditions were relaxed in (\cite{S87}). It was also extended to processes with RCLL paths in convex polyhedra (\cite{DI91}), for general semimartingales in convex regions (\cite{AL91}), and in non-smooth domains (\cite{C92}). 

We consider processes with both normal and oblique reflections, with slightly different sets of assumptions that ensure the respective SDER is well-defined. Let $S\subset \mathbb{R}^d$ be a bounded convex set, and $\mathcal{N}_x$ be the set of all inward unit normal vectors at $x\in\partial S$: $\mathcal{N}_x=\cup_{r>0}\mathcal{N}_{x,r},~\mathcal{N}_{x,r}=\{\textbf{n}\in\mathbb{R}^d:|\textbf{n}|=1, B(x-r\textbf{n},r)\cap S=\emptyset\}$, where $B(z,r)=\{y\in\mathbb{R}^d:|y-z|<r, z\in\mathbb{R}^d, r>0\}$ is the ball with radius $r>0$ around a point $z\in\mathbb{R}^d$. Consider the following assumptions on $S$ from (\cite{S93}, \cite{LS84}, \cite{S87}):
\begin{itemize}
	\item[(A)] There exists a constant $r_0>0$ such that $\mathcal{N}_x=\mathcal{N}_{x,r_0}\neq\emptyset$ for every $x\in\partial S$; 
	\item[(B)] There exist constants $\delta>0, \beta\geq 1$ such that for every $x\in\partial S$ and for every $\textbf{n}\in\cup_{y\in B(x,\delta)\cap \partial S}\mathcal{N}_y$, there exists a unit vector $\textbf{e}_x$ such that $\langle \textbf{e}_x,n\rangle \geq \frac{1}{\beta}$, where $\langle\cdot,\cdot\rangle$ denotes the inner product.
\end{itemize}
Condition (A) guarantees the existence of a unit normal vector at each point on the boundary with a uniform sphere about itself. This condition is satisfied by the boundedness assumption on $S$ (c.f. \cite{AL91}, \cite{T79}) and $r_0=+\infty$ by the convexity of $S$ (c.f. Remark 1(iii), \cite{S93}). Condition (B) is equivalent to defining a uniform cone on the interior of $S$ at each boundary point. Call $V\in\mathcal{D}(S)$ a solution to the {\bf stochastic differential equation with normal reflection} if there exists an associated boundary process $L\in\mathcal{D}(\mathbb{R})$ such that $L(0)=0$, and $L(t)=\int_{0}^{t}\textbf{1}_{\{V(s)\in\partial S\}}dL(s)$, that is $L$ is equal to its own total boundary variation on $[0,t]$, and for all $t\ge 0$  $V(t)$  satisfies the equation
\begin{equation}\label{eq-V-norm}
	V(t)=V(0)+\int_{0}^{t}\mu(V(s))ds+\int_{0}^{t}\sigma(V(s))dB(s)+\int_{0}^{t}\!\!\int_{\mathcal{M}}\gamma(V(s-),y)N(ds,dy)+\int_{0}^{t}\textbf{n}(V(s))dL(s),\quad V(0) \in \overline S.
\end{equation}  
If we assume that $\mu,\sigma,$ and $\gamma$ are bounded on $S$, it was shown  (Theorem 5, \cite{S93}) by a step function approximation that there exists a solution of the SP with normal reflection and a unique strong solution to the SDER (\ref{eq-V-norm}). 

We also consider oblique reflections at the boundary as they are used in chemical reaction approximations to match the combined behaviour of all the reactions that are active on the boundaries (c.f. \cite{LW19}). An analogous reflected process exists if we impose some additional assumptions. Let $S\subset \mathbb{R}^d$ be a bounded simply connected region  with a smooth, connected and orientable boundary $\partial S$ (the condition on the boundary can be relaxed, c.f. Remark 5 \cite{MR85}). Define a twice continuously differentiable vector field $\rho$ in a neighborhood of $\bar{S}$ such that $-\rho(x)\cdot n(x)\geq \epsilon > 0, \forall x\in\partial S$. Assume that for all $x\in\bar{S},y\in\mathcal{M}, x+\gamma(x,y)\in\bar{S}$, that is, all jumps from $\bar S$ remain in $\bar S$ (c.f. Section 2, \cite{MR85}). Call $V\in\mathcal{D}(S)$ a solution to the {\bf stochastic differential equation with oblique reflection} if there exists a continuous associated boundary process $L$ such that  $L(0)=0$ and $L(t)=\int_{0}^{t}\textbf{1}_{\{V(s)\in\partial S\}}dL(s)$, and for all $t\ge 0$  $V(t)$  satisfies the equation
\begin{equation}\label{eq-V-arb}
V(t)=V(0)+\int_{0}^{t}\mu(V(s))ds+\int_{0}^{t}\sigma(V(s))dB(s)+\int_{0}^{t}\!\!\int_{\mathcal{M}}\gamma(V(s-),y)N(ds,dy)+\int_{0}^{t}\rho(V(s))dL(s),\quad V(0) \in \overline S.
\end{equation}
The existence and uniqueness of a solution to \eqref{eq-V-arb} was shown in (\cite{MR85}) by first establishing a solution for normal reflection by a penalization argument and then constructing a diffeomorphism between $\bar S$ and a closed unit ball where $\rho$ is mapped to an outward normal vector to extend existence and uniqueness to appropriate oblique reflections as well. Note that by setting $\rho\equiv \textbf{n}$,  we get the normally reflected process \eqref{eq-V-norm} as a special case of  \eqref{eq-V-arb}, but with more restrictions imposed on $S$ and $L$ for the oblique case. 

Solutions to SDERs have been established under different assumptions on its driving processes and domains. The first results (\cite{T79}) are for a diffusion process in a convex domain with normal reflection, extended by (\cite{LS84}) for a diffusion with normal and arbitrary reflections in their domain. It was further shown in (\cite{DI93}) that obliquely reflected diffusions also exist in non-smooth domains with corners. For RCLL processes, it was shown that there exists a unique solution in the positive half-space (\cite{CElKM80}), and more recently (\cite{LS03}) established  existence of arbitrarily reflected semimartingales allowing jumps at the boundary of the domain, using convergence of approximating processes in the S-topology.     \\   

We next introduce the additive functional of the reflected jump-diffusion process $V$. Let $f$ be a bounded continuous function on $\bar S$, and define the {\bf additive functional} $\Lambda$ as 
\begin{equation}\label{eq-L}
\Lambda(t)=\int_0^t f(V(s))ds +\int_0^t f(V(s))\, dL^c(s).
\end{equation}
where $L^c$ is the continuous part of the associated boundary process $L$,
\begin{equation}\label{eq-lc}
	L^c(t) = L(t) - \sum\limits_{0<s\leq t:\Delta L(s)\neq 0}\Delta L(s).
\end{equation} 
Using a sequence of continuous functions approximating the step function $f=\mathbf 1_A$,  we can recover (by convergence of associated additive functionals) 
the total occupation time of $V$ in an arbitrary $A\subset S$. For $A\subset \partial S$, we can also  recover the total variation over $A$ of the continuous part of the associated boundary process $L^c$. Many other quantities of interest may be studied by an appropriate choice of $f$. \\

\subsection{Main result}
Our main result considers the large deviation principle for  the additive functional $\Lambda$ in terms of a PIDE for calculating its logarithmic moment generating (spectral radius) function.

We make the following assumptions on the transition semigroup of the reflected jump-diffusion $V$ (which by uniqueness of solutions to the SDER is a Markov process). These  ensure $V$ is a Feller process whose  occupation measure converges to the invariant measure exponentially fast (c.f. \cite{DVIV83} p.187 I-II, \cite{FSS87} (2.1)-(2.3)). 
Assume there exists a probability measure $\mu$ on $\bar S$ such that for all $t>0$:
\vspace{.2cm}
\begin{itemize}
			\item[(i)] the semigroup $T_t$ of $V$ has a density $p(t,x,y)$ relative to $\mu$: $T_tu(x)= \int_{\bar S} u(y) p(t,x,y) \mu(dy)$;
			\item[(ii)] $0<a(t)\leq p(t,x,y)$ for some strictly positive function $a(t)$ and for all $x,y\in \bar S$; and
			\item[(iii)] $\lim\limits_{x\to x_0}||p(t,x,\cdot)-p(t,x_0,\cdot)||_{L^1(S,\mu)}=0$.		\end{itemize}
\vspace{.2cm}

\begin{theorem}[{\bf PIDE and exponential martingale}]\label{thm-mgale} For all $\theta\in \R$ there exists a unique, up to a constant, positive twice continuously differentiable function  $u_{\theta}(\cdot)\in C_+^2(\bar S)$ on $\bar S$, and a scalar $\psi_{\theta}\in\R$ such that the pair $(u_{\theta}(\cdot),\psi_{\theta})$ satisfies the partial-integro differential equation 
		\begin{align}\label{eq-PIDE}
		\sum_{i=1}^d\partial_{x_i}u_{\theta}(x)\mu_i(x)+\frac12 \sum_{i,j=1}^d&\partial^2_{x_ix_j}u_{\theta}( x)(\sigma\sigma^T)_{ij}(x)+\int_{\mathcal{M}}[
		u_{\theta}(r(x,y))-u_{\theta}(x)]\nu_x(dy)\nonumber\\
		&\;\;+u_{\theta}(x)(\theta f(x)-\psi_{\theta})=0, \; \forall x\in S
		\end{align}
		subject to boundary conditions (if $L^c\not\equiv 0$) 
		\begin{align}\label{eq-PIDEbdry}
		\theta f(x)u_{\theta}(x)+\sum_{i=1}^{d}\partial_{x_i}u_{\theta}(x)\rho_{i}(x)=0, \; \forall x\in\partial S
		\end{align}
		and such that \begin{equation}\label{eq-mgale}M_{\theta}(t)=e^{\theta \Lambda(t)-\psi_{\theta}t} u_{\theta}( V(t))\end{equation} is a martingale.
	\end{theorem}

\begin{proof}
We start by identifying the equations that $u_\theta$ and $\psi_\theta$ would need to satisfy in order for $M_\theta$ to be a martingale.
We use of the following notation for ease of exposition. For $x\in S$, $y\in \mathcal{M}$, let $[x+\gamma(x,y)]_{_{\partial S}}$ denote the projection onto $\partial S$ resulting from a jump exceeding the region $S$. Define $r:S\times \mathcal{M}\to \overline S$ 
$$
r(x,y)=\left\{\begin{array}{ll}
x+\gamma(x,y)& \textrm{if }x+\gamma(x,y)\in S\\
\left[x+\gamma(x,y)\right]_{_{\partial S}}& \textrm{if }x+\gamma(x,y)\notin S\\
\end{array}\right..
$$ 
As $V$ is a semi-martingale we can apply It\^o's formula  (Theorem II.33, \cite{P13}) to $M_{\theta}(t)$. Since $L$ has paths of finite variation on finite intervals we have that $[\Lambda,V]^c(t)=[\Lambda,\Lambda]^c(t)=0$ and $[V,V]^c(t)=\int_{0}^{t}\sigma\sigma^T(V(s))ds$ where $[\cdot,\cdot]_t^c$ denotes the continuous part of the quadratic covariation between the two processes. Ito's formula on  \eqref{eq-mgale} gives 
		\begin{align*}
			M_{\theta}(t)-&M_{\theta}(0) = \int_{0}^{t}e^{\theta\Lambda(s-)-\psi_{\theta}s-}(-\psi_{\theta})u_{\theta}(V(s-))ds+\int_{0}^{t}e^{\theta\Lambda(s-)-\psi_{\theta}s-}\theta u_{\theta}(V(s-))d\Lambda^c(s)\\
			&+\int_{0}^{t}e^{\theta\Lambda(s-)-\psi_{\theta}s-}\sum_{i=1}^{d}\partial_{x_i}u_{\theta}(V(s-))dV^c(s)\\
			&+\frac{1}{2}\int_{0}^{t}e^{\theta\Lambda(s-)-\psi_{\theta}s-}\sum_{i,j=1}^{d}\partial_{x_ix_j}^2u_{\theta}(V(s-))(\sigma\sigma^T)_{ij}(V(s-))ds\\
			&+\sum_{0<s\leq t:\Delta V(s)\neq 0}e^{\theta\Lambda(s-)-\psi_{\theta}s-}\left(
			u_{\theta}(V(s))-u_{\theta}(V(s-))\right)
		\intertext{Replacing $\Lambda^c$, $V^c$ with their definitions, compensating the jumps, collecting all like integrators, and replacing the summation by it's Poisson representation, we have}
			M_{\theta}(t)-&M_{\theta}(0)
			=\int_{0}^{t}e^{\theta\Lambda(s-)-\psi_{\theta}s-}\biggl(\left(\theta f(V(s-))-\psi_{\theta}\right)u_{\theta}(V(s-))+\sum_{i=1}^{d}\partial_{x_i}u_{\theta}(V(s-))\mu_i(V(s-))\\
			&+\frac{1}{2}\sum_{i,j=1}^{d}\partial_{x_ix_j}^2u_{\theta}(V(s-))(\sigma\sigma^T)_{ij}(V(s-))
			+\int_{\mathcal{M}}\left(
			u_{\theta}(r(V(s-),y))-u_{\theta}(V(s-))\right)\nu_{V(s-)}(dy)\biggr)ds\\
			&+\int_{0}^{t}e^{\theta\Lambda(s-)-\psi_{\theta}s-}\biggl(\theta f(V(s-))u_{\theta}(V(s-))+\sum_{i=1}^{d}\partial_{x_i}u_{\theta}(V(s-))\rho_{i}(V(s-))\biggr)dL^c(s)\\
			&+\int_{0}^{t}e^{\theta\Lambda(s-)-\psi_{\theta}s-}\sum_{i=1}^{d}\partial_{x_i}u_{\theta}(V(s-))\sum_{j=1}^{d}\sigma_{ij}(V(s-))dB_j(s)\\
			&+\int_{0}^{t}\int_{\mathcal{M}}e^{\theta\Lambda(s-)-\psi_{\theta}s-}\left(u_{\theta}(r(V(s-),y))-u_{\theta}(V(s-))\right)\left(N(ds,dy)-\nu_{V(s-)}(dy)ds\right)
			\end{align*}
			For any pair $(u_{\theta}(\cdot),\psi_{\theta})$  satisfying the  equations (\ref{eq-PIDE})-(\ref{eq-PIDEbdry}) we then get
			\begin{align*}
			M_{\theta}(t)-&M_{\theta}(0)=\int_{0}^{t}e^{\theta\Lambda(s-)-\psi_{\theta}s-}\sum_{i=1}^{d}\partial_{x_i}u_{\theta}(V(s-))\sum_{j=1}^{d}\sigma_{ij}(V(s-))dB_j(s)\\
			&+\int_{0}^{t}\int_{\mathcal{M}}e^{\theta\Lambda(s-)-\psi_{\theta}s-}\left(u_{\theta}(r(V(s-),y))-u_{\theta}(V(s-))\right)\left(N(ds,dy)-\nu_{V(s-)}(dy)ds\right)
		\end{align*}
			For $u_{\theta}(\cdot)\in C_+^2(\bar S)$ the term $\sum_{i=1}^d\partial_{x_i}u_{\theta}(\cdot)$ is bounded on $\bar S$. Likewise $\sigma$ is assumed Lipschitz continuous and so it is bounded on $\bar S$. Since $f$ is continuous, we also have 
			\[\sup\limits_{0<s\leq t}\left|e^{\theta\Lambda(s)}\right|\leq e^{\theta t\cdot \|f\|_{\infty,\bar S}+Ct\cdot \|f\|_{\infty,\bar S}}<\infty\]
			for some constant $C$ satisfying $L^c(t)\leq Ct$ and where $\|f\|_{\infty,\bar S}=\sup\{|f(x)|:\; x\in \bar S \}$. Then we have
			\[\int_0^t \int_S \mathbb{E}\big[e^{\theta\Lambda(s-)}[u_\theta(r(V(s-),y))-u_\theta(V(s-))]\big] \nu_{V(s-)}(dy)ds<2te^{\theta\tilde C t\|f\|_{\infty,\bar S}}\|u_\theta\|_{\infty, \bar S}< \infty, \]
			for some constant $\tilde C$ large enough.	We can conclude that both the integral with respect to Brownian motion and the integral with respect to the compensated Poisson random measure are martingales since any uniformly bounded local martingale is a martingale. Hence their sum with $M_{\theta}(0)$, i.e. $M_{\theta}(t)$, is a martingale.

For the existence and uniqueness of $u_\theta$ and $\psi_\theta$ satisfying (\ref{eq-PIDE})-(\ref{eq-PIDEbdry}) we next summarize the arguments from Theorem 4.1 in (\cite{FSS87}). 
Let $\tilde T_t$ be the strongly continuous
semigroup  defined by \[\tilde T_t u(x)=\mathbb{E}[e^{\theta \Lambda(t)} u(V(t))|V(0)=x].\] 
Let $\mathcal{L}$ be the second order linear operator on $u\in C^2(\bar S)$ given by
		\begin{equation}\label{eq-operator}
			\mathcal{L}u(x) = \sum_{i=1}^d\partial_{x_i}u(x)\mu_i(x)+\frac12 \sum_{i,j=1}^d\partial^2_{x_ix_j}u(x)(\sigma\sigma^T)_{ij}(x)+\int_{\mathcal{M}}[u(r(x,y))-u(x)]\nu_x(dy).
		\end{equation}
Similarly to our earlier calculation, It\^o's formula gives
		\begin{align*}
			e^{\theta\Lambda(t)}u(V(t))-&u(V(0))=\int_{0}^{t}e^{\theta\Lambda(s-)}\big(\mathcal{L}u(V(s-))+ \theta f(V(s-))u(V(s-))\big)ds\\
			&+\int_{0}^{t}e^{\theta\Lambda(s-)}\biggl(\theta f(V(s-))u(V(s-))+\sum_{i=1}^{d}\partial_{x_i}u(V(s-))\rho_{i}(V(s-))\biggr)dL^c(s)\\
			&+\int_{0}^{t}e^{\theta\Lambda(s-)}\sum_{i=1}^{d}\partial_{x_i}u(V(s-))\sum_{j=1}^{d}\sigma_{ij}(V(s-))dB_j(s)\\
			&+\int_{0}^{t}\int_{\mathcal{M}}e^{\theta\Lambda(s-)}\left(u(r(V(s-),y))-u(V(s-))\right)\left(N(ds,dy)-\nu_{V(s-)}(dy)ds\right).
		\end{align*}
Taking expectations, and as the last two integrals are martingales, we get that $\,\tilde T_t\,u - u = \int_{0}^{t}\tilde T_s\,\tilde{\mathcal{L}}u\,ds$ \;holds with 
the infinitesimal generator of $\tilde T_t$ given by the second order operator $\tilde{\mathcal{L}}=\mathcal{L} +\theta f$, on the set of functions  \[{D}(\tilde{\mathcal{L}})=\{u\in C^2(\bar{S}): \theta f(x)u(x)+\sum_{i=1}^d\partial_{x_i}u(x)\rho_{i}(x)=0, \forall x\in \partial S\}.\]
Assumptions (i)-(iii) on the semigroup of $V$ imply that  the semigroup $\tilde T_t$ also has a density $\tilde{p}(t,x,y)$ with respect to $\mu$ for each $t>0$ and satisfies (ii)-(iii), which can be easily verified. For each $t>0$ this ensures existence and uniqueness (\cite{K64} Theorems 2.8, 2.10) of a pair $(u_{\theta,t}, \psi_{\theta,t})$ of a positive continuous function $u_{\theta,t}$ on $\bar S$ and $\psi_{\theta,t}\in \mathbb R$, such that
\[\tilde T_t u_{\theta,t}=e^{\psi_{\theta,t}}u_{\theta,t}\;\mbox{ and }\;  \max_{x\in \bar S}u_{\theta,t}(x)=1\]
One can further show (c.f. the argument in \cite{FSS87} p.7), that there exist a probability measure $\eta_\theta$ and $\psi_\theta\in \mathbb R$ independent of $t$, such that for all $v\in C(\bar S)$  and all $t>0$
\[\int_{x\in\bar S} \tilde T v(x)\eta_\theta(dx)=e^{\lambda_\theta\cdot t}\int_{x\in \bar S}v(x)\eta_\theta(dx).\] 
Integrating $\tilde T_t u_{\theta,t}$ with respect to $\eta_\theta$, together with uniqueness of $\psi_{\theta,t}$ then imply that $\psi_{\theta,t}=\psi_\theta\, t$. Iterating the semigroup property gives $\tilde T_{nt} u_{\theta,t}=e^{\psi_\theta \,{nt}}u_{\theta,t}$, and then uniqueness of $u_{\theta,t}$ implies $u_{\theta,t}=u_{\theta,1}$ $\forall t$ rational. Since $\tilde T_t u_{\theta,1}\in {D}(\tilde{\mathcal L})$ and is positive, we have $u_{\theta,1}\in {D}_+(\tilde{\mathcal L})$. Furthermore, $\tilde{\mathcal L}u_{\theta,1}=\psi_\theta u_{\theta,1}$, so that uniqueness of $u_{\theta,1}$ implies uniqueness, up to a constant,  of this positive eigenfunction for $\tilde{ \mathcal L}=\mathcal L + \theta f$. 

A positive eigenfunction and eigenvalue of the problem 
$\tilde{\mathcal{L}}u_\theta(x)= \psi_\theta u_\theta(x)$
with eigenfunction satisfying the boundary constraint  of ${D}(\tilde{\mathcal L})$ are  the desired solution $(u_\theta,\psi_\theta)$ of \eqref{eq-PIDE}-\eqref{eq-PIDEbdry}.	
\end{proof}

\begin{remark}
The proof above identifies, for each $\theta\in\mathbb R$ and $f\in C(\bar S)$ bounded, the constant $\psi_\theta$ as the principal eigenvalue of the operator $\tilde{\mathcal L}=\mathcal L +\theta f$, where $\mathcal L$ is the generator of the reflected jump-diffusion $V$ as in \eqref{eq-operator}. The next result will show it is the limit of log-moment generating function of the additive process $\Lambda$. Donsker-Vardhan theory (\cite{DVIII76,DVIV83}) then implies it can also be expressed in terms of a variational problem (c.f. Thm 1.1. in \cite{FSS87}).
\end{remark}

Viewed as a function of $\theta\in\mathbb{R}$ the eigenvalue $\psi_\theta$ is used to establish the large deviations for $\Lambda$ defined as follows. $\{\Lambda(t)\}_{t\ge 0}$ is said to satisfy the {\bf large deviation principle (LDP)} with rate $t$ and good rate function $I$ if:\, $I\not\equiv \infty$; $I$ has compact level sets (so is lower-semicontinuous); $$\limsup_{t\to\infty}\frac 1t \log \mathbb P(\Lambda(t)\in C)\le \inf _{x\in C} I(x), \; \forall C\subset \mathbb{R}\text{ closed},$$ and $$\limsup_{t\to\infty}\frac 1t \log \mathbb P(\Lambda(t)\in O)\le \inf _{x\in O} I(x), \; \forall O\subset \mathbb{R}\text{ open.}$$
Using the result of Thm 4.1 in (\cite{FSS87}) one can prove this LDP using their Thm 1.1, or one can use the G\"artner-Ellis theorem (c.f. \cite{DZ98} Thm 2.3.6, \cite{dH00} Thm V.6) as below. 

\begin{corollary}[\bf logarithmic moment generating function and G\"artner-Ellis]\label{cor-ge} 
For all $\theta\in \mathbb R$ the logarithmic moment generating function of $\Lambda$ satisfies
		\begin{equation}\label{eq-limitgf}
		\lim_{t\to\infty}\frac{1}{t}\log{\mathbb E\left[e^{\theta\Lambda(t)}\right]}=\psi_{\theta},
		\end{equation}
and $\Lambda$ satisfies the LDP with rate $t$ and good rate function $I=\psi^*$ given by the Legendre transform of $\psi$
		\begin{equation}\label{eq-rateftion}
		\psi^*(x)=\sup_{\theta}[\theta x-\psi_{\theta}].
		\end{equation}

	\end{corollary}
\begin{proof}
Since $M_{\theta}(t)$ is a martingale $\mathbb{E}\left[e^{\theta\Lambda(t)-\psi_{\theta}t}u_{\theta}(V(t))\right]=u_{\theta}(V(0)).$ By the positivity and boundedness of $u_{\theta}(x)$, it follows that
		\[
		e^{-\psi_{\theta}t}\mathbb{E}[e^{\theta\Lambda(t)}]\inf\limits_{x\in \bar S} u_{\theta}(x)\leq u_{\theta}(V(0))\leq e^{-\psi_{\theta}t}\mathbb{E}[e^{\theta\Lambda(t)}]\sup\limits_{x\in \bar S} u_{\theta}(x),
		\]
		which implies
		\begin{equation}\label{psi-estim}
			\frac{1}{t}\log\mathbb{E}\left[e^{\theta\Lambda(t)}\right]+\frac{1}{t}\log\inf\limits_{x\in \bar S} u_{\theta}(x)\leq \psi_{\theta}+\frac{1}{t}\log u_{\theta}(V(0)) \leq\frac{1}{t}\log\mathbb{E}\left[e^{\theta\Lambda(t)}\right]+\frac{1}{t}\log\sup\limits_{x\in \bar S} u_{\theta}(x).
		\end{equation}
		Therefore, \eqref{eq-limitgf} follows by taking $t\to\infty$.

The G\"artner-Ellis theorem requires that the following properties of $\psi_\theta$ as a function of $\theta$ are satisfied:  \\
(a) $0\in \mbox{int}(\mathcal D_\psi)$ where $D_{\psi}=\{\theta\in\mathbb{R}:\psi_{\theta}<\infty\}$; (b) $\psi$ is lower semi-continuous in $\theta$; (c) $\psi_\theta$ is differentiable with respect to $\theta$ on $\mbox{int}(\mathcal D_\psi)$; and (d) $D_\psi=\R$ or $\lim_{\theta\in\mathcal D_\psi\to \partial D_\psi}|\nabla\psi_{\theta}|=\infty$.

	By Theorem \ref{thm-mgale},  there exists a finite eigenvalue $\psi_{\theta}$ for each fixed $\theta\in\mathbb{R}$, so $D_{\psi}=\mathbb{R}$.  Since $\psi_0=0$, the origin is in the interior of $D_{\psi}$. 			
	
		Since $\psi_\theta$ is the eigenvalue associated with a unique positive eigenfunction of  $\mathcal{L}+\theta f$, its differentiability with respect to $\theta$ follows by use of the implicit function theorem (c.f. \cite{FSS87} p.2, \cite{KM03} Prop 4.8).
	
	Let $D_{\psi,c}=\{\theta\in\mathbb{R}:\psi_{\theta}\leq c\}$ denote sublevel sets of $\psi_\cdot$ with respect to $\theta$, and let $\theta\in \bar{D}_{\psi,c}$. There exists a sequence $\{\theta_n\}_{n\geq0}\subset D_{\psi,c}$ such that $\theta_n\uparrow\theta$. Then, \[\lim\limits_{n\to\infty}\psi_{\theta_n}=\lim\limits_{n\to\infty}\lim\limits_{t\to\infty}\frac{1}{t}\log\mathbb{E}\left[e^{\theta_n\Lambda(t)}\right]=\lim\limits_{t\to\infty}\frac{1}{t}\log\mathbb{E}\left[e^{\theta\Lambda(t)}\right]=\psi_{\theta}\leq c,\] where the interchange of limits and the passage of the limit through the expectation follows from the monotone convergence theorem and monotonicity in $n$. Then $\theta\in D_{\psi,c}$ which implies $\bar{D}_{\psi,c}\subset D_{\psi,c}$. As all sublevel sets of $\psi_\theta$ are closed, we have that $\psi_{\theta}$ is lower semicontinuous in $\theta$. 
	\end{proof}\\
	
	\begin{remark}
		Our main goal was characterizing the logarithmic moment generating function $\psi_\theta$ by solving a boundary value PIDE, from which 
		we can also obtain the long term mean and variance for $\Lambda$. From the martingale $M_\theta$ \eqref{eq-mgale} we have
		 \begin{equation}\label{eq-mgaleE}\mathbb{E}\left[e^{\theta\Lambda(t)} u_{\theta}(V(t))\right]=e^{\psi_{\theta}t}u_{\theta}(V(0))\end{equation}
	Assuming $u_\theta$ is $C^2$ as a function of $\theta$, taking derivatives with respect to $\theta$, using dominated convergence gives
	\begin{align*}
			\frac{1}{te^{\psi_\theta t}u_{\theta}(V(0))}\mathbb{E}\left[\Lambda(t)e^{\theta\Lambda(t)}u_{\theta}(V(t))+e^{\theta\Lambda(t)}\frac{d}{d\theta}u_{\theta}(V(t))\right]=\frac{d\psi_{\theta}}{d\theta}+\frac{1}{t}\frac{\frac{d}{d\theta}u_{\theta}(V(0))}{u_{\theta}(V(0))}
		\end{align*} 
		Evaluating at $\theta = 0$, and using the fact that at $\theta = 0$ $(u_0(\cdot),\psi_0)\equiv (1,0)$ solves the PIDE (\ref{eq-PIDE})-(\ref{eq-PIDEbdry}), we get
	\begin{align*}
			\frac{1}{t}\mathbb{E}\left[\Lambda(t)+\frac{d}{d\theta}u_{\theta}(V(t))\big|_{\theta=0}\right]=\frac{d\psi_{\theta}}{d\theta}\big|_{\theta=0}+\frac{1}{t}{\frac{d}{d\theta}u_{\theta}(V(0))\big|_{\theta=0}}
		\end{align*} 	
Taking the limit as $t\to\infty$, one gets for the long term mean of $\Lambda$ that
		\begin{align*}
		\lim\limits_{t\to\infty}\frac{1}{t}\mathbb{E}\left[\Lambda(t)\right]=\frac{d\psi_{\theta}}{d\theta}\big|_{\theta=0}.
		\end{align*}
		Taking second derivatives in \eqref{eq-mgaleE} with respect to $\theta$, one similarly gets for the long term variance of $\Lambda$ that
				\begin{align*}
		\lim\limits_{t\to\infty}\frac{1}{t}\mathbb{V}\left[\Lambda(t)\right]=\frac{d^2\psi_{\theta}}{d\theta^2}\big|_{\theta=0}.
		\end{align*}
	\end{remark}

\subsection{One dimensional reflected jump-diffusion}

In the case of a jump-diffsuion in one dimension, much of the general theory simplifies to more explicit formulae. In particular, there is only one possible direction of reflection at each boundary, and there is an explicit formula for the SP mapping taking $X$ to $(V,L)$ that provides a direct construction of the reflected process $V$. 

Without loss of generality let us assume $S=[0,b]$ for some $b<\infty$. Since $L$ is a process on $\partial S=\{0,b\}$ we can decompose $L=L_0-L_b$ where we define $L_0$ as the associated boundary process at $0$ and $L_b$ as the associated boundary process at $b$, in the sense that, for all $t>0$:
\begin{align} \label{eq-SP1d}
\int_0^t \!\!\!\!\mathbf 1_{\{V(s)>0\}}dL_0(s)=0, \;\; \int_0^t \!\!\!\!\mathbf 1_{\{V(s)<b\}}dL_b(s)=0.
\end{align} 
The explicit formula for $(V,L)$  was constructed in (\cite{KLRS07, KLRS08}) providing $V(t)=\mathcal L_{[0,b]}(X)(t)$ via the following Skorokhod map  
\begin{equation}\label{eq-SM1dV}
\mathcal L_{[0,b]}(X)(t):=X(t)-\bigg[(X(0)-b)^+\land \inf_{0\le u\le t}X(u) \bigg] \lor \sup_{0\le s\le t}\bigg[(X(s)-b)\land \inf_{s\le u\le t}X(u) \bigg]
\end{equation}
using the notation $(x-b)^+=(x-b)\lor 0$.
It was also shown in  (\cite{KLRS07}) that the associated boundary processes satisfy
\begin{equation}\label{eq-SM1dL}
L_0(t)=\sup_{0\le s\le t}(L_b(s)-X(s))^+,\;\, L_b(t)=\sup_{0\le s\le t}(X(s)+L_0(s)-b)^+,
\end{equation}
and that this formula with $V=X+L_0-L_b$ is equivalent to the one in \eqref{eq-SP1d}.

With $S=[0,b]$, the additive functional becomes
\begin{equation}\label{eq-L1}
\Lambda(t)=\int_0^t f(V(s))ds +f(0) L^c_0(t)+f(b) L^c_b(t)
\end{equation}
where $L_0^c$ and $L_b^c$ are the continuous parts of the increasing processes $L_0$ and $L_b$ respectively.
We can recover them by using continuous approximations to the function $f=\mathbf{1}_{0}$ to get $\Lambda(t)=L^c_0(t)$; and continuous approximations to $f=\mathbf{1}_{b}$ to get $\Lambda(t)=L^c_b(t)$.

By Theorem \ref{thm-mgale} we have that $M_{\theta}(t)=e^{\theta \Lambda(t)-\psi_{\theta}t}$ is a martingale, and for each $\theta\in\mathbb R$ there exists a unique positive function $u_{\theta}(x)$  and constant $\psi_{\theta}$ satisfying the PIDE
\begin{align}\label{eq-PIDEd1}
\partial_{x}u_{\theta}(x)\mu(x)+\frac12 &\partial^2_{xx}u_{\theta}( x)\sigma^2(x)+\int_{\mathcal{M}}[u_{\theta}(r(x,y))-u_{\theta}(x)]\nu_x(dy)+u_{\theta}(x)(\theta f(x)-\psi_{\theta})=0
\end{align}
subject to the boundary conditions, if $L^c \not\equiv  0$
\begin{align}\label{eq-PIDEbdryd1}
\theta u_{\theta}(0)f(0)+\rho_0\partial_{x}u_{\theta}(0)=0, \; \; \theta u_{\theta}(b)f(b)-\rho_b\partial_{x}u_{\theta}(b)=0,
\end{align}
Without the jump part, the resulting PDE would be more straightforward to solve numerically. We construct a numerical approximation scheme for $u_{\theta}(x)$ based on the PIDE \eqref{eq-PIDEd1}-\eqref{eq-PIDEbdryd1} that allows for jumps (c.f. Appendix for full details), and test it on two examples of PIDEs whose solution we can derive analytically.

\section{Examples}The partial integro-differential equation \eqref{eq-PIDE}-\eqref{eq-PIDEbdry} for the pair $(u_{\theta}(x),\psi_{\theta})$ can be solved analytically only in a few special cases; even then the expression for $\psi$ is implicit. In general, one has to use numerical methods to solve this PIDE. We next provide a numerical scheme to approximate its solution using $d=1$ and $S=[0,b]$. 

Fix $\theta\in\mathbb{R}$, and let $\tilde{\mathcal{L}}$ be the second order linear operator defined  by $\tilde{\mathcal{L}}u_\theta(x)-\psi_{\theta}u_{\theta}(x)=0$ in (\ref{eq-PIDEd1}) on the subset of $C^2([0,b])$ functions $D(\tilde{\mathcal{L}})$ defined by the boundary constraint (\ref{eq-PIDEbdryd1}). Since the limit of the logarithmic moment generating function  is the eigenvalue $\psi_{\theta}$ of the operator $\tilde{\mathcal{L}}$ with corresponding eigenfunction $u_{\theta}(\cdot)$, we will numerically solve the eigenvalue problem $\tilde{\mathcal{L}}u_{\theta}(x)=\psi_{\theta}u_{\theta}(x)$. We subdivide $\bar S=[0,b]$ into $N+1$ equal sub-intervals and approximate $\tilde{\mathcal{L}}$ as a matrix. We treat the derivatives by finite differences and the integral by composite trapezoidal quadrature on intervals with continuous support and as a weighted sum on intervals with discrete support. The boundary conditions are approximated by forward and backward finite-difference schemes and substituted into the matrix where appropriate. Further details on the numerical method are contained in the Section 4 appendix.

We now apply this numerical scheme  in two special cases for which we can also derive analytical implicit solutions to the limiting cumulant generating function (Sections  \ref{RefHomDiff}-\ref{RefBD}), in order to test our numerical scheme. We then apply the numerical estimation to our application of interest: a biochemical reaction model and its jump-diffusion approximation (Section \ref{example-biochem}).   

\subsection{Reflected Brownian motion with drift}\label{RefHomDiff}
Let $X$ be a standard Brownian motion in $d=1$ on $[0,b]$ with drift $\mu$, diffusion $\sigma^2$ (and $\nu\equiv 0$) and $V$ be this process normally reflected at the boundaries. Let $\Lambda=L^c_0$ be the local time of X at $0$ (so $f=$ continuised version of $\mathbf{1}_{0}$ as described below (\ref{eq-lc})). Then \eqref{eq-PIDEd1}-\eqref{eq-PIDEbdryd1} becomes the partial differential equation 
with boundary constraints \begin{align*}
\frac12 \sigma^2 \partial^2_{xx}u_{\theta}(x)+\mu \partial_xu_{\theta}(x)-\psi_{\theta}u_{\theta}(x)&=0\\
\theta u_{\theta}(0)+\partial_x u_{\theta}(0)=0, \; u_{\theta}(b)=1,\; \partial_xu_{\theta}(b)&=0.
\end{align*}
(as $u_\theta$ is unique only up to a constant, we are allowed to chose its value at $x=b$).
When the characteristic polynomial  $\beta^2+\frac{2\mu}{\sigma^2}\beta-\frac{2\psi_{\theta}}{\sigma^2}=0$ has repeated roots, $\psi_{\theta}=-\mu^2/2\sigma^2$, then 
\begin{align*}
u_{\theta}(x)=e^{\frac{\mu(b-x)}{\sigma^2}}\left(\frac{-b\mu+\sigma^2+\mu x}{\sigma^2}\right),\;
\theta=-\frac{\mu^2 b}{\sigma^2(\sigma^2-b\mu)}.
\end{align*}
When the polynomial has real roots, $\psi_{\theta}>-\mu^2/2\sigma^2$, then for $\alpha=\sqrt{\mu^2+2\sigma^2\psi_{\theta}}$
\begin{align*}
u_{\theta}(x)=e^{\frac{b(\mu-\alpha)-x(\alpha+\mu)}{\sigma^2}}\left(\frac{(\alpha-\mu)e^{\frac{2\alpha b}{\sigma^2}}+(\alpha+\mu)e^{\frac{2\alpha x}{\sigma^2}}}{2\alpha}\right),\;
\theta=\frac{2\sigma^2\psi_{\theta}e^{\frac{2\alpha b}{\sigma^2}}}{\sigma^2((\alpha-\mu)e^{\frac{2\alpha b}{\sigma^2}}+(\alpha+\mu))}.
\end{align*}
When the polynomial has complex roots, $\psi_{\theta}<-\mu^2/2\sigma^2$, then for $\alpha=\sqrt{-(\mu^2+2\sigma^2\psi_{\theta})}$
\begin{align*}
u_{\theta}(x)=e^{\frac{\mu(b-x)}{\sigma^2}}\left(\frac{\alpha\cos\left(\frac{\alpha(b-x)}{\sigma^2}\right)-\mu\sin\left(\frac{\alpha(b-x)}{\sigma^2}\right)}{\alpha}\right),\;
\theta=\frac{2\sigma^2\psi_{\theta}\sin(\frac{\alpha b}{\sigma^2})}{\sigma^2(\mu \sin(\frac{\alpha b}{\sigma^2})-\alpha\cos(\frac{\alpha b}{\sigma^2}))}.
\end{align*}
In the special case of $\mu=0,\sigma^2=1$:   $\psi_{\theta}$ is given implicitly by $\theta=\sqrt{2\psi_{\theta}}\tanh(b\sqrt{2\psi_{\theta}})$ with $\psi_0=0$. A similar result was obtained by other analytic methods in (\cite{FKZ15}).

For the numerics, we set $b=1$ and approximate $f$ by using continuous linear piecewise  functions $f(x) = (1-(N+1)x)\textbf{1}_{\{x<\frac{1}{N+1}\}}$ on $[0,1]$, for the choice of mesh steps $N$ to be defined. Figure \ref{fig:numericsincrnbm} shows the convergence of the proposed numerical scheme to the analytical solution as the number of mesh steps are increased from $N=10$ to $N=110$ for various values of $\theta$. Table \ref{rhbm} compares the analytical result of $\psi_{\theta}$ against its numerical approximation, $\hat{\psi_{\theta}}$ for $\theta$ close to $0$. In this approximation the interval $S=[0,1]$ is subdivided $N+1=1001$ sub-intervals.

\begin{figure}[H]
	\begin{floatrow}
		\ffigbox{
			{\includegraphics[width=0.95\linewidth]{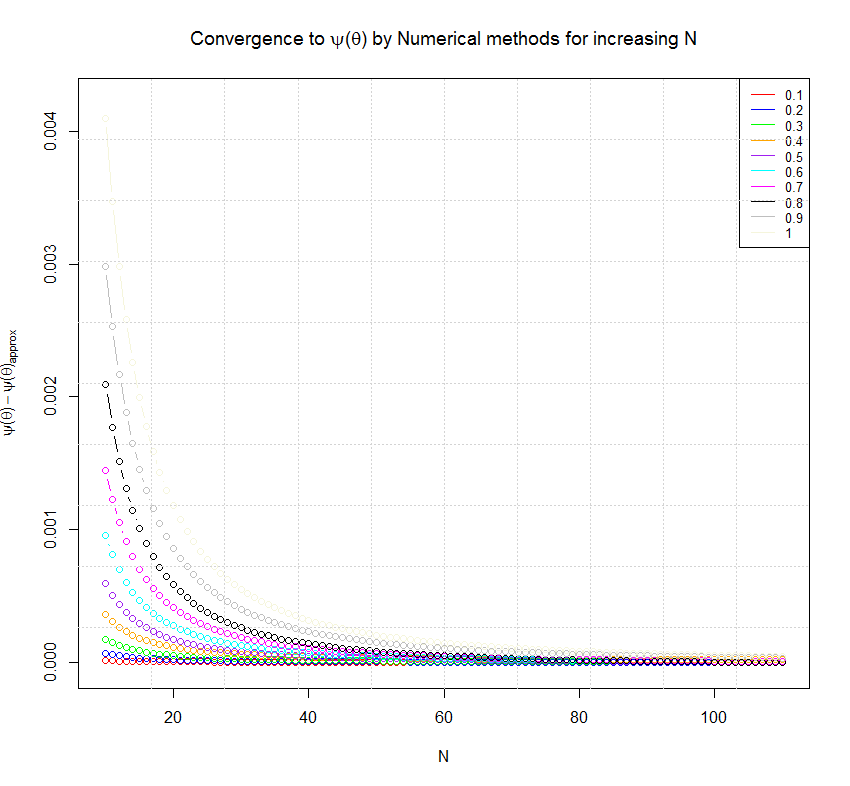} }
		}{
			\caption{Convergence of $\hat{\psi_{\theta}}$ for $\theta\in\{\frac{i}{10}\}_{1}^{10}$ as\\  $\{N\}_{10}^{110}$ increases for a reflected standard BM.}\label{fig:numericsincrnbm}
		}
		\capbtabbox{
			\begin{tabular}{|c|c|c|c|c|}
				\hline
				$\theta$&$\psi_{\theta}$& $\hat{\psi_{\theta}}$ & $|\psi_{\theta}-\hat{\psi_{\theta}}|$  \\
				\hline
				$0$& $0$ & $3.761\times10^{-10}$ & $3.761\times10^{-10}$  \\
				\hline
				$0.001$&$5.003\times10^{-4}$ &$5.002\times10^{-4}$  & $9.816\times10^{-8}$ \\
				\hline
				$0.002$&$1.001\times10^{-3}$ & $1.001\times10^{-3}$ &$3.942\times10^{-7}$  \\
				\hline
				$0.003$&$1.502\times10^{-3}$ & $1.502\times10^{-3}$ & $8.859\times10^{-7}$  \\
				\hline
				$0.004$& $2.004\times10^{-3}$&$2.003\times10^{-3}$  &$1.573\times10^{-6}$    \\
				\hline
				$0.005$&$2.507\times10^{-3}$ &$2.504\times10^{-3}$  & $2.457\times10^{-6}$  \\
				\hline
				$0.006$&$3.010\times10^{-3}$ &$3.006\times10^{-3}$  & $3.536\times10^{-6}$  \\
				\hline
				$0.007$&$3.513\times10^{-3}$ &$3.508\times10^{-3}$  & $4.809\times10^{-6}$  \\
				\hline
				$0.008$&$4.017\times10^{-3}$ &$4.011\times10^{-3}$  & $6.277\times10^{-6}$  \\
				\hline
				$0.009$& $4.521\times10^{-3}$& $4.514\times10^{-3}$ &$7.937\times10^{-6}$   \\
				\hline
				$0.01$&$5.027\times10^{-3}$ &$5.017\times10^{-3}$  & $9.792\times10^{-6}$ \\
				\hline
			\end{tabular}\vspace{0.85cm}
		}{
			\caption{Comparison between $\psi_{\theta}$ and $\hat{\psi_{\theta}}$ when $\theta$ is near 0 for a reflected standard BM.}\label{rhbm}
		}
	\end{floatrow}
\end{figure}

\subsection{Reflected birth-death process}\label{RefBD}
Let $X$ be a pure birth-death process in $d=1$ on $S=\{0,1,2,...,b\}$ with overall jump rate $\lambda$, so $\nu_x(\pm 1)=\frac\lambda 2$ (and $\mu=\sigma\equiv 0$) and $V$ be this process but it reflects on itself at the boundaries: $\nu_x(\pm 1)=\frac\lambda 2$ for $x\in\{1,\dots,b-1\}$ and $\nu_0(+1)=\nu_0(0)=\frac\lambda 2, \nu_b(-1)=\nu_b(0)=\frac\lambda 2$. Let $\Lambda=\int_0^t \mathbf{1}_{V_s\in [0,1)}ds$ and $f(x) = \textbf{1}_{0\le x<1-\frac{1}{N+1}}+(N+1)(1-x)\textbf{1}_{\{1-\frac{1}{N+1}\le x<1\}}$ approximating the function $\mathbf{1}_{ [0,1)}$. Then  \eqref{eq-PIDEd1}-\eqref{eq-PIDEbdryd1} becomes the recurrence relation equation
\begin{align*}
-(\lambda+\psi_{\theta}-\theta f(x))u_{\theta}(x) +\frac\lambda 2\left(u_{\theta}(r(x,1))\right)+\frac\lambda 2\left(u_{\theta}(r(x,-1))\right)=0\quad\\
\Longleftrightarrow\quad
\left\{\begin{array}{ll}
-(\frac \lambda 2+\psi_{\theta}-\theta)u_{\theta}(0)+\frac\lambda 2u_{\theta}(1)=0,~~~&x=0\\
-(\lambda+\psi_{\theta})u_{\theta}(x)+\frac\lambda 2u_{\theta}(x+1)+\frac\lambda 2u_{\theta}(x-1)=0,~~~&x=\{1,2,...,b-1\}\\
-(\frac \lambda 2+\psi_{\theta})u_{\theta}(b)+\frac\lambda 2u_{\theta}(b-1)=0,~~~&x=b.\\
\end{array}\right.
\end{align*}
with sole constraint $u_{\theta}(b)=1$ ($L^c\equiv 0$, since $X$ is a pure jump process).

The recurrence equation on $x=\{1,2,...,b-1\}$ is a linear difference equation and may be solved by standard methods. Namely, let $u_{\theta}(x)=\beta^x$ then the characteristic equation is
\begin{align*}
-(\lambda+\psi_{\theta})\beta^x+\frac{\lambda}{2}\beta^{x+1}+\frac{\lambda}{2}\beta^{x-1}=0 \iff\beta^2-\frac{2}{\lambda}(\lambda+\psi_{\theta})\beta+1=0.
\end{align*}
Solving for $\beta$, we have $\beta = \frac{1}{\lambda}\left((\lambda+\psi_\theta)\pm\sqrt{\psi_{\theta}(2\lambda+\psi_{\theta})}\right).$
Consequently,
\begin{align*}
u_{\theta}(x)=\begin{cases}
&C_1\left(\frac{1}{\lambda}(\lambda+\psi_{\theta})\right)^x+C_2x\left(\frac{1}{\lambda}(\lambda+\psi_{\theta})\right)^x,~~~\psi_{\theta}\in\{0,-2\lambda\}\\
&C_1\left(\frac{1}{\lambda}\left((\lambda+\psi_\theta)+\sqrt{\psi_{\theta}(2\lambda+\psi_{\theta})}\right)\right)^x+C_2\left(\frac{1}{\lambda}\left((\lambda+\psi_\theta)-\sqrt{\psi_{\theta}(2\lambda+\psi_{\theta})}\right)\right)^x~~~\text{otherwise.}
\end{cases}
\end{align*}
Using the boundary condition $u_{\theta}(b)=1$, we can derive from the third recurrence relation that
$u_{\theta}(b-1)=1+\frac{2}{\lambda}\psi_{\theta},$ and using this in the second recurrence relation with $x=b-1$, we have $u_{\theta}(b-2)= 1 + \frac{6}{\lambda}\psi_{\theta}+\frac{4}{\lambda^2}\psi_{\theta}^2.$ Now we may solve for $C_1$ and $C_2$ to find that
\begin{align*}
u_{\theta}(x)=\begin{cases}
&\frac{1}{\lambda^3}\left(\frac{\lambda+\psi_{\theta}}{\lambda}\right)^{1-b+x}\left(\lambda^3+(5b-3-5x)\lambda^2\psi_{\theta}+10(b-1-x)\lambda\psi_{\theta}^2+4(b-1-x)\psi_{\theta}^3\right),~~~\psi_{\theta}\in\{0,-2\lambda\}\\
&\frac{1}{2}\left( \left(\frac{\lambda+\psi_{\theta}-\sqrt{\psi_{\theta}(2\lambda+\psi_{\theta})}}{\lambda}\right)^{x-b}
+ \frac{\psi_{\theta}\left( \frac{\lambda+\psi_{\theta}-\sqrt{\psi_{\theta}(2\lambda+\psi_{\theta})}}{\lambda} \right)^{x-b}}{\sqrt{\psi_{\theta}(2\lambda+\psi_{\theta})}}\right.\\ 
&\left.~~~~~+\left(\frac{\lambda+\psi_{\theta}+\sqrt{\psi_{\theta}(2\lambda+\psi_{\theta})}}{\lambda}\right)^{x-b}
+\frac{\psi_{\theta}\left( \frac{\lambda+\psi_{\theta}+\sqrt{\psi_{\theta}(2\lambda+\psi_{\theta})}}{\lambda} \right)^{x-b}}{\sqrt{\psi_{\theta}(2\lambda+\psi_{\theta})}}  \right),~~~\text{otherwise.}
\end{cases}
\end{align*}

We equate the first recurrence equation to the second recurrence equation when $x=1$ to derive the identity
\begin{align*}
u_{\theta}(0)=\frac{\frac{\lambda}{2}}{\frac{\lambda}{2}+\psi_{\theta}-\theta}u_{\theta}(1)=\frac{2}{\lambda}(\lambda+\psi_{\theta})u_{\theta}(1)-u_{\theta}(2),
\end{align*}
which can now be used to solve for $\psi_{\theta}$ implicitly. Namely for $\psi_{\theta}\not\in\{0,-2\lambda\}$, $\theta=\frac{A}{B}$ where  
\begin{align*}
A =& \psi_{\theta}\left(\lambda\left(\left(\frac{-\sqrt{\psi_{\theta}(2\lambda+\psi_{\theta})}+\lambda+\psi_{\theta}}{\lambda}\right)^b-\left(\frac{\sqrt{\psi_{\theta}(2\lambda+\psi_{\theta})}+\lambda+\psi_{\theta}}{\lambda}\right)^b\right)\right.\\
&\left.+\psi_{\theta}\left(\left(\frac{-\sqrt{\psi_{\theta}(2\lambda+\psi_{\theta})}+\lambda+\psi_{\theta}}{\lambda}\right)^b-\left(\frac{\sqrt{\psi_{\theta}(2\lambda+\psi_{\theta})}+\lambda+\psi_{\theta}}{\lambda}\right)^b\right)\right.\\
&\left.-\sqrt{\psi_{\theta}(2\lambda+\psi_{\theta})}\left(\left(\frac{-\sqrt{\psi_{\theta}(2\lambda+\psi_{\theta})}+\lambda+\psi_{\theta}}{\lambda}\right)^b+\left(\frac{\sqrt{\psi_{\theta}(2\lambda+\psi_{\theta})}+\lambda+\psi_{\theta}}{\lambda}\right)^b\right)\right),\\
B =& \psi_{\theta}\left(\left(\frac{-\sqrt{\psi_{\theta}(2\lambda+\psi_{\theta})}+\lambda+\psi_{\theta}}{\lambda}\right)^b-\left(\frac{\sqrt{\psi_{\theta}(2\lambda+\psi_{\theta})}+\lambda+\psi_{\theta}}{\lambda}\right)^b\right)\\
&-\sqrt{\psi_{\theta}(2\lambda+\psi_{\theta})}\left(\left(\frac{-\sqrt{\psi_{\theta}(2\lambda+\psi_{\theta})}+\lambda+\psi_{\theta}}{\lambda}\right)^b+\left(\frac{\sqrt{\psi_{\theta}(2\lambda+\psi_{\theta})}+\lambda+\psi_{\theta}}{\lambda}\right)^b\right).
\end{align*}

For the purposes of the numerical approximations, we will arbitrarily set $\lambda=50$ and $b=3$. With this choice, $\psi_{\theta}$ is given implicitly by $\theta=\frac{\psi_{\theta}(62500+6250\psi_{\theta}+150\psi_{\theta}^2+\psi_{\theta}^3)}{15625+3750\psi_{\theta}+125\psi_{\theta}^2+\psi_{\theta}^3}$. Similarly, using $u_{\theta}(x)$ when $\psi_{\theta}\in\{0,-2\lambda\}$, we deduce that with this choice of $\lambda$ and $b$ we have that $\psi_0=0$ and $\psi_{-50}=-100$. Given that $S$ is discrete, numerical approximations of $\psi_{\theta}$, $\hat\psi_\theta$, can be made by solving for the largest real eigenvalue of the matrix
\[
\begin{bmatrix}
-\frac{50}{2}+\theta & \frac{50}{2} & 0 & 0 \\
\frac{50}{2} & -50 & \frac{50}{2} & 0 \\
0 & \frac{50}{2} & -50 & \frac{50}{2} \\
0 & 0 & \frac{50}{2} & -\frac{50}{2}
\end{bmatrix}
\]
for each $\theta$. Table \ref{rbd} compares the analytical result of $\psi_{\theta}$ against its numerical approximation, $\hat{\psi_{\theta}}$, for various values of $\theta$ close to $0$. 

\begin{figure}[H]
		\capbtabbox{
			\begin{tabular}{|c|c|c|c|c|}
				\hline
				$\theta$&$\psi_{\theta}$& $\hat{\psi_{\theta}}$ & $|\psi_{\theta}-\hat{\psi_{\theta}}|$  \\
				\hline
				$0$& $0$ & $-6.476\times10^{-317}$ & $6.476\times10^{-317}$  \\
				\hline
				$0.001$&$2.503\times10^{-4}$ &$2.500\times10^{-4}$  & $3.410\times10^{-7}$ \\
				\hline
				$0.002$&$5.007\times10^{-4}$ & $5.000\times10^{-4}$ &$6.646\times10^{-7}$  \\
				\hline
				$0.003$&$7.510\times10^{-4}$ & $7.501\times10^{-4}$ & $9.706\times10^{-7}$  \\
				\hline
				$0.004$& $1.001\times10^{-3}$&$1.000\times10^{-3}$  &$1.259\times10^{-6}$    \\
				\hline
				$0.005$&$1.252\times10^{-3}$ &$1.250\times10^{-3}$  & $1.530\times10^{-6}$  \\
				\hline
				$0.006$&$1.502\times10^{-3}$ &$1.500\times10^{-3}$  & $1.784\times10^{-6}$  \\
				\hline
				$0.007$&$1.752\times10^{-3}$ &$1.750\times10^{-3}$  & $2.020\times10^{-6}$  \\
				\hline
				$0.008$&$2.003\times10^{-3}$ &$2.001\times10^{-3}$  & $2.238\times10^{-6}$  \\
				\hline
				$0.009$& $2.253\times10^{-3}$& $2.251\times10^{-3}$ &$2.439\times10^{-6}$   \\
				\hline
				$0.01$&$2.503\times10^{-3}$ &$2.501\times10^{-3}$  & $2.623\times10^{-6}$ \\
				\hline
			\end{tabular}\vspace{0.85cm}
		}{
			\caption{Comparison between $\psi_{\theta}$ and $\hat{\psi_{\theta}}$ when $\theta$ is near 0 for a reflected birth-death process.}\label{rbd}
		}
\end{figure}

\subsection{Biochemical reaction model}\label{example-biochem}
Our main motivating example comes from jump-diffusion approximation of a biochemical reaction model. The full model is a pure Markov jump process tracking the amount $X_A(t)$ and $X_B(t)$ of molecular species $A$ and $B$ within a cell that also undergoes cellular growth and division. The simplified representation of reactions between $A$ and $B$ is:
\begin{align*} 
A   & \stackrel{\tilde\kappa^{10}_{-1}}{\to} B \\
B    & \stackrel{\tilde\kappa^{01}_{1}}{\to} A \\
A+B  & \stackrel{\tilde\kappa^{11}_{-1}}{\to} 2B \\
2A+B & \stackrel{\tilde\kappa^{21}_{1}}{\to} 3A 
\end{align*}
where all external factors are captured by reaction constants $\tilde\kappa^{ij}_k$. At the time of cellular division an assignment of one half of doubled molecules results in a Bernoulli($\frac12$) random error of $\,\pm 1$ in the amount of species $A$ compensated by species $B$, occurring at a rate that is proportional to their product and a division constant $\tilde\gamma$. The jump Markov process for the evolution of this system has the generator:
\begin{align*}\mathcal{G}g(x_A,y_B)&=\tilde\kappa^{10}_{-1}x_A[g(x_A-1,y_B+1)-g(x_A,y_B)]+\tilde\kappa^{01}_{1}y_B[g(x_A+1,y_B-1)-g(x_A,y_B)]\\&+\tilde\kappa^{11}_{-1}x_Ay_B[g(x_A-1,y_B+1)-g(x_A,y_B)]+\tilde\kappa^{21}_{1}x_A^2y_B[g(x_A+1,y_B-1)-g(x_A,y_B)]\\&+\frac12\tilde\gamma x_Ay_B[g(x_A-1,y_B+1)-g(x_A,y_B)]+\frac12\tilde\gamma x_Ay_B[g(x_A+1,y_B-1)-g(x_A,y_B)]\end{align*}
for $g\in\mathcal C(\mathbb{N}\times \mathbb{N})$. This reaction and division dynamics has two relevant features:

(1) There is a conservation law in the total sum of species $A$ and $B$, and letting $n$ denote the  initial overall total of both species, $X_A=\frac{x_A}{n}, X_B=\frac{y_B}{n}$ denote the proportions (out of $n$) of species $A, B$ respectively, all the reactions preserve the initial total $X_A(0)+X_B(0)=1$ so that $X_B(t)=1-X_A(t), \forall t>0$ reduces the model to $d=1$.
This allows us to express the rates of reactions, which are proportional to the product of source types masses and the chemical reaction constants, where the latter are assumed to scale as $\tilde\kappa=n^{\nu-1}\kappa$ with $\nu$=number of sources in the reaction.  The generator of the process $X_A=\frac{x_A}{n}$ for $g\in\mathcal C([0,1])$ is:
\begin{align*}\mathcal{G}_ng(x)&=n\big(\kappa^{10}_{-1}x+\kappa^{11}_{-1}x(1-x)\big)\big[g(x-\!\frac1n)-g(x)\big]+n\big(\kappa^{01}_{1}(1-x)+\kappa^{21}_{1}x^2(1-x)\big)\big[g(x+\!\frac1n)-g(x)\big]\\&+\frac12\gamma_n x(1-x)\big[g(x-\!\frac1n)-g(x)\big]+\frac12\gamma_n x(1-x)[g(x+\!\frac1n)-g(x)]\end{align*}
showing the overall reaction rates of decreasing and increasing proportions of $A$ as $nr_-(x)=n\kappa_{-1}^{10}x+n\kappa_{-1}^{11}x(1-x)$ and $nr_+(x)=n\kappa_1^{01}(1-x)+n\kappa_1^{21}x^2(1-x)$, respectively. The division rates for both increasing and decreasing proportions of $A$ are $\xi_n(x)=\frac12\gamma_n x(1-x)$, where the relationship of $\gamma_n$ to  $n$ will be explored in the two approximations of the jump Markov chain to follow. 

(2) The long term dynamics exhibits a form of noise induced bistability in the proportion of species $A$, under appropriate assumptions on the constants $\{\kappa^{\cdot\cdot}_\cdot\}$, c.f. ($\star$) below. The rate of change of the mean is: 
\begin{equation}\label{eq-CRNdrift}\frac{d}{dt}E\left[X_A(t)|X_A(t)=x\right]=\mu(x):=-{\kappa}^{10}_{-1} x +{\kappa}^{01}_{1}(1-x) - {\kappa}^{11}_{-1} x(1-x) + {\kappa}^{21}_{1} x^2(1-x),~ x\in[0,1]\end{equation} and since $\mu(x)$ is a cubic, assuming  ($\star$) that it has all real roots  in $[0,1]$, then the dynamics of $E[X_A(t)]$ has two stable equilibria and one unstable equilibrium point creating potential barriers on either side of the domain of attraction of the two equilibria. Freidlin-Wentzel theory (\cite{FW98} Ch 6) for path properties of Markov processes with $O(n)$ rates and $O(\frac1n)$ jump sizes  imply that this process will spend most of its time in the stable equilibria with rare transitions between small neighbourhoods around them created by perturbations due to randomness in the system. 
The occupation measure process will reflect this and increasingly concentrate at the deterministic stable points. For more details on sample path properties on finite time intervals of biochemical reaction models with division errors (c.f. Sec 3.3. of \cite{McSP14}).

To estimate the mean local time at the two stable equilibria, as well as the large deviations away from this mean, we numerically solve the PIDE for the limiting cumulant generating function $\psi_{\theta}$, which again reduces to solving for the largest eigenvalue of an $(n+1)\times(n+1)$ matrix.
Set $n=100$, and $\kappa_1^{01}=1,\kappa_1^{21}=\frac{32}{3},\kappa_{-1}^{10}=1,\kappa_{-1}^{11}=\frac{16}{3}$ (for which  ($\star$) is satisfied as the cubic has all real roots), then the two stable equilibria  are $x_1=0.25$ and $x_2=0.75$ on $S=[0,1]$, 
and we will define $f(x)$ as the continuised version of $\textbf{1}_{x\in B_{1/(N+1)}(0.25)\cup B_{1/(N+1)}(0.75)}$, so that the additive functional $\Lambda(t)$ measures the time spent {\bf around} the two stable equilibria.

The jump Markov process parameters are $\nu_x(+\frac{1}{n})=\xi_n(x)+nr_+(x)$ and $\nu_x(-\frac{1}{n})=\xi_n(x)+n r_-(x)$  (and $\mu=\sigma^2\equiv 0$). On the two boundaries the rates of the reaction dynamics and division errors at $x\in\{0,1\}$ have only inward jumps ($r_-(0)= r_+(1)=0$, $\xi_n(0)=\xi_n(1)=0$), so no additional reflection is needed to keep the process within $S=[0,1]$, and $V_n=X^n_A$.

When $\gamma_n= n$, the rate of division errors $\xi_n(x)=\frac12nx(1-x)$ is of the same order as the rate  of reactions. Hence, a rigorous approximation  of $X^n_A(t)$ in terms of $n$ can be made in terms of the constrained Langevin process (c.f. \cite{LW19, AHLW19}) which is a reflected diffusion $V_n$  on $[0,1]$ with small noise 
\begin{align}\label{eq-Vcrn-cle}
	dV_n(t) = \mu(V_n(t))dt +\! \frac{1}{\sqrt{n}}\sqrt{r_+(V_n(t))+r_-(V_n(t))+V_n(t)(1-V_n(t))}dW(t)+\!\frac{1}{\sqrt{n}}\big(\rho_0 dL_0(t)-\rho_1 dL_1(t)\big),
\end{align}
Its drift is equal to  $\mu(x)=-\kappa_{-1}^{10}x+\kappa_1^{01}(1-x)-\kappa_{-1}^{11}x(1-x)+\kappa_1^{21}x^2(1-x)$ from \eqref{eq-CRNdrift} as the division error is unbiased, the diffusion coefficent $\sigma^2(x)=r_+(x)+r_-(x)+x(1-x)$ is equal to the sum of all rates as the square of all jumps are of size $(\pm \frac1n)^2=\frac1{n^2}$, and the reflection directions are $\rho_0 = \rho_1 = 1$ on associated boundary processes $L_0$ and $L_1$ respectively. Large deviation theory for path properties of small noise diffusions (\cite{DZ98} Ch 5, \cite{FW98} Ch 5) also implies this process spends most of its time in the neighbourhood of  stable equilibria with  rare excursions transitioning between them.  The occupation measure of the process will concentrate near the stable equilibria in the long term limit. To estimate the numerical solution for the limiting cumulant generating function $\psi_{\theta}$ of the same additive functional $\Lambda(t)$ as above, we set the jump-diffusion parameters to $\mu(x), \sigma^2(x)$ as above (and $\nu_x\equiv 0$), and we define a sequence of continuous linear piecewise functions $f$ by:
\begin{align}
f(x)=\begin{cases}
0 & 0\leq x < 0.25-\frac{1}{N+1}\\
(N+1)x+1-0.25(N+1) & 0.25-\frac{1}{N+1} \leq x < 0.25\\
-(N+1)x+1+0.25(N+1) & 0.25 \leq x < 0.25+\frac{1}{N+1}\\
0 & 0.25+\frac{1}{N+1} \leq x < 0.75 - \frac{1}{N+1}\\
(N+1)x+1-0.75(N+1) & 0.75 - \frac{1}{N+1} \leq x < 0.75\\
-(N+1)x+1+0.75(N+1) & 0.75 \leq x < 0.75 + \frac{1}{N+1}\\
0 & 0.75 + \frac{1}{N+1} \leq x \leq 1.
\end{cases}\nonumber
\end{align}

We use  $\psi_{\theta,JMP}$ to denote the numerical solution to the PIDE for the limiting cumulant generating function $\psi_{\theta}$ of the jump Markov process $X^n_A(t)$, and $\psi_{\theta,JDA}^N$ for the reflected diffusion approximation $V_n(t)$ in \eqref{eq-Vcrn-cle}, where we use  $n=100$ in the first set of results and $n=1000$ in the second. 
We empirically verify the stability of the numerical approximation in mesh size in Figure \ref{group1}: surfaces (a) and (b) approximate $\psi_{\theta,JDA}^N$ for $0\leq\theta\leq1$, $\{N\}_{50}^{150}$ and $\{N\}_{500}^{1500}$, respectively; the bottom plots (c) and (d) compare the approximations of $\psi_{\theta,JMP}$ (in red) and $\psi_{\theta,JDA}^N$ (in blue), for $n=N=100$ and $n=N=1000$, respectively. 

\noindent To estimate the mean $\psi'(0)$ and variance $\psi''(0)$ of local times for fixed $n=N=1000$ we consider centered finite difference approximations using $\theta=\{-0.01,0,0.01\}$ and obtain:
\begin{align*}
&\psi_{0,JMP}' \approx 1.2\times10^{-3},\quad  \psi_{0,JMP}'' \approx1.6\times10^{-5};\\
&\psi_{0,JDA}' \approx 2.6\times10^{-3},\quad  \psi_{0,JDA}''\approx 5.2\times10^{-5}.
\end{align*}

	\begin{figure}[h]
	\begin{adjustwidth}{-1in}{-1in}
	\centering
	\subfloat[\centering $\psi_{\theta,JDA}^N$ plot for $n=100$, $0\leq\theta\leq1$ as  $\{N\}_{50}^{150}$ increases (reflected diffusion)]{{\includegraphics[scale=0.4]{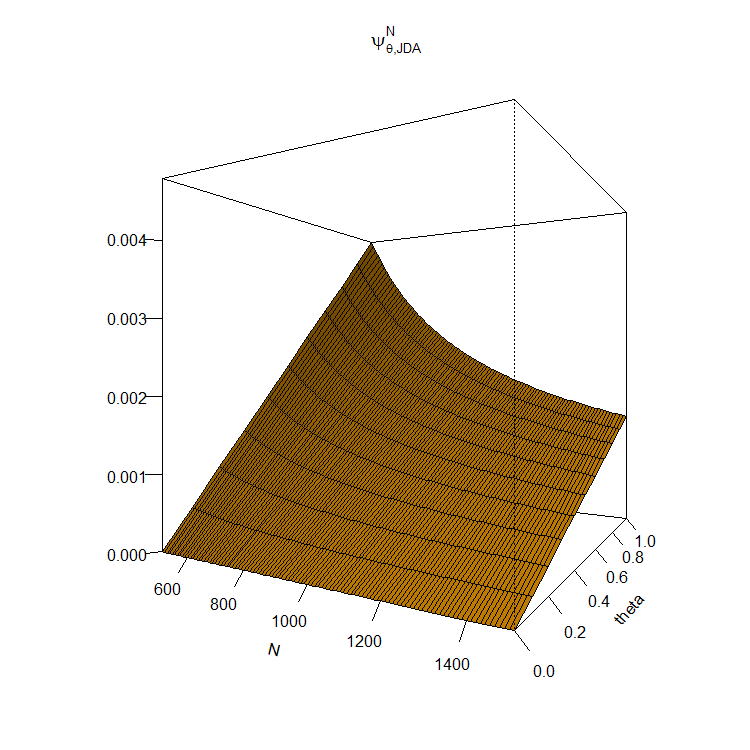} }}
	\quad
	\subfloat[\centering $\psi_{\theta,JDA}^N$ plot for $n=1000$, $0\leq\theta\leq1$ as $\{N\}_{500}^{1500}$ increases (reflected diffusion)]{{\includegraphics[scale=0.4]{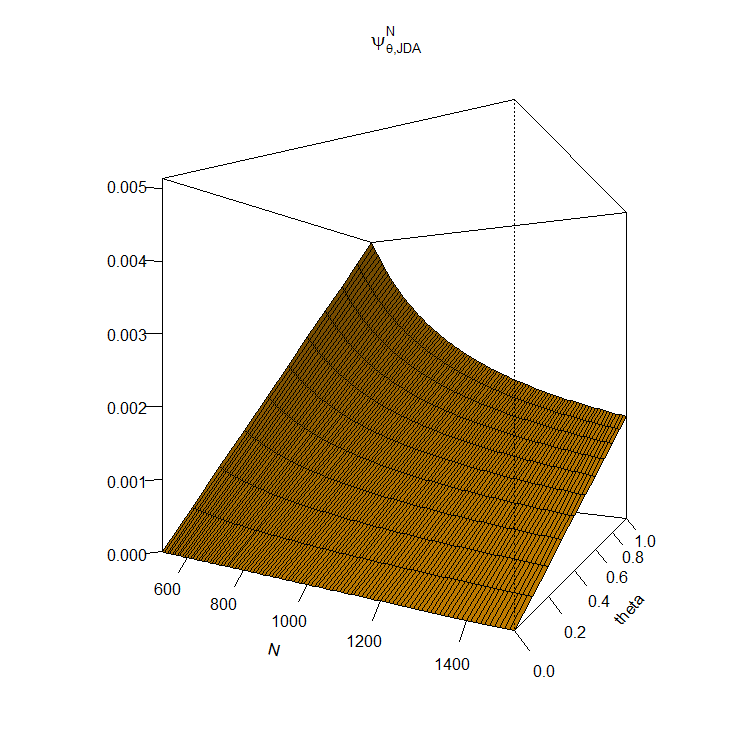} }}
	\quad
	\subfloat[\centering $\psi_{\theta,JMP}$ and $\psi_{\theta,JDA}^N$ plots for $n=N=100$, $0\leq\theta\leq100$.]{{\includegraphics[scale=0.3]{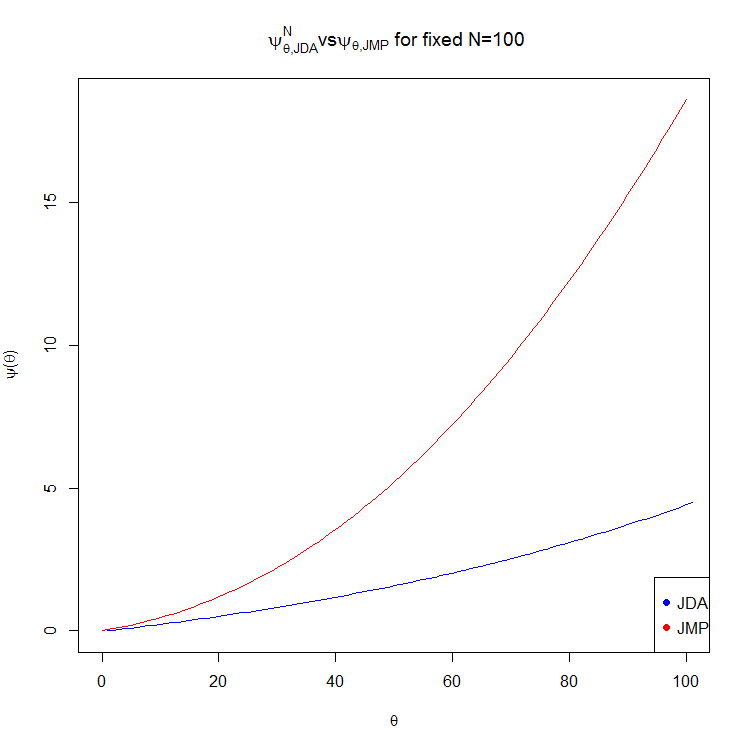} }}
	\quad
	\subfloat[\centering $\psi_{\theta,JMP}$ and $\psi_{\theta,JDA}^N$ plots for $n=N=1000$, $0\leq\theta\leq100$.]{{\includegraphics[scale=0.3]{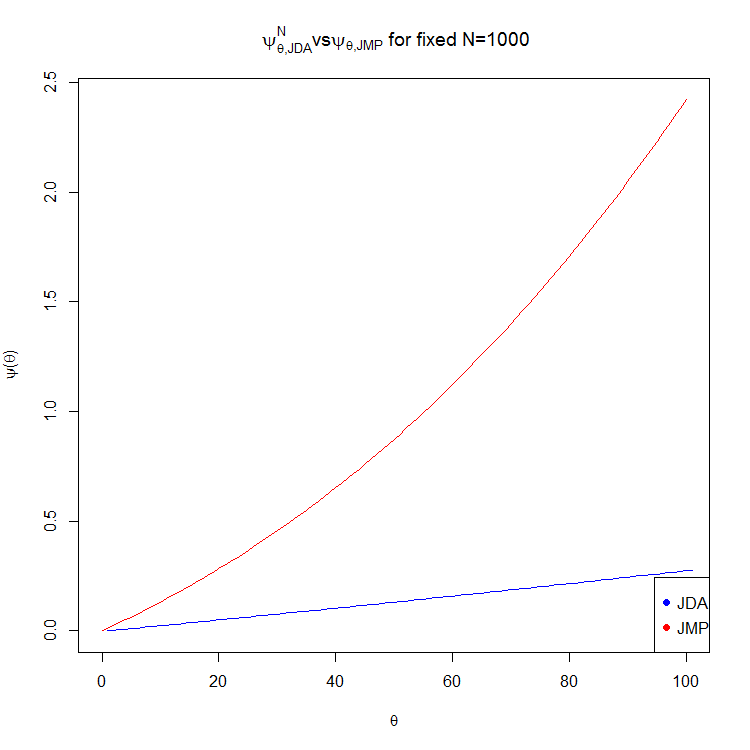} }}
	\caption{Numerical estimate of the limiting cumulant generating function $\psi_{\theta}$ for the long-term local time of $\{0.25,0.75\}$ of the jump Markov process $\psi_{\theta,JMP}$ and of the reflected diffusion $\psi_{\theta,JDA}^N$  (case $\gamma_n=n$).}
	\label{group1}
	\end{adjustwidth}
	\end{figure}

Comparing the results for $n=100$ with $n=1000$ shows the numerical approximation is more stable as the rescaling parameter $n$ increases, since then the magnitude of noise decreases in both processes.
However, the long-term mean local time differs in the two models regardless of the increase in scaling parameter. This is in contrast with the finite time results which say that the paths of the two processes become closer in $n$, but can be reasoned by the fact that the sup-norm of the path difference features a multiplying constant that is a function of the length of the time interval (\cite{K78, KKP14}), and that our results are based on taking limits as the time of integration goes to infinity, and not as the noise size goes to zero.
Our numerical results indicate that the mean of the local time at equilibria are smaller for the jump Markov model than for the reflected diffusion, indicating a tighter long term concentration of the stationary distribution of the reflected JDA at equilibria compared to that of the JMP. This is in full agreement with the results of \cite{McSP14}  Thm~3.1 which say that the functional path large deviation rate of jumps between the two stable equilibria is higher for the jump Markov model than for the reflected diffusion. Since more frequent transitions result in a less concentrated measure our present observation follows.

We next consider the case when division errors occur at higher rate $\gamma_n\gg n$ (e.g.\,$\gamma_n=10n^2$) than the rates of reactions. The $\,\pm \frac1n$ errors from division now dominate the noise and it is more suitable to approximate the model instead by a process in which only the reaction contributions are modeled by a constrained Langevin equation with small noise while the error division contributions are still modeled by pure jumps. The jump-diffusion model in part reduces computation, but also allows a comparison with the approximate diffusion model in the case $\gamma_n=n$. We now define $\tilde V_n$ to be the following  reflected jump-diffusion 
\begin{align}\label{eq-Vcrn-clejump}
		d\tilde V_n(t) &=& \mu(\tilde V_n(t))dt + \frac{1}{\sqrt{n}}\sqrt{r_+(\tilde V_n(t))+r_-(\tilde V_n(t))} dW(t)\pm \frac1n dY^{\xi(\tilde V_n)}_\pm(t)+\frac{1}{\sqrt{n}}\big(\rho_0 dL_0(t)-\rho_1 dL_1(t)\big),
\end{align}
where each $Y^{\xi(\tilde V_n)}_-$ and $Y^{\xi(\tilde V_n)}_+$ are counting processes with intensity measure $\xi_n(\tilde V_n)=\frac12\gamma_n\tilde V_n(1-\tilde V_n)$.
We use $\gamma_n=10n^2$ and present the results for the local time at the same two equilibria $\{0.25,0.75\}$ in Figure \ref{group2}.

\begin{figure}[h]
	\begin{adjustwidth}{-1in}{-1in}
		\centering
		\subfloat[\centering $\psi_{\theta,JDA}^N$ plot for $n=100$, $0\leq\theta\leq1$ as  $\{N\}_{50}^{150}$ increases (reflected jump diffusion)]{{\includegraphics[scale=0.4]{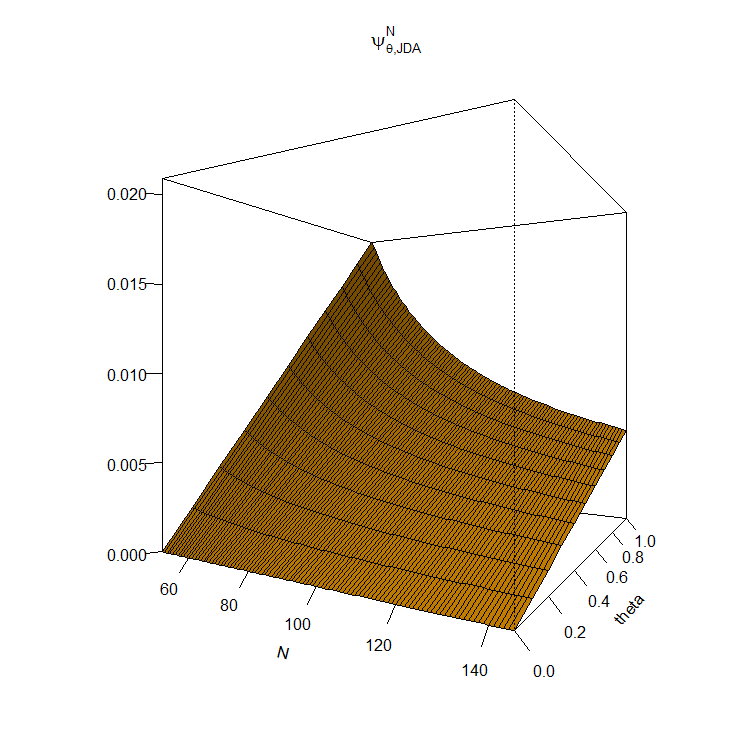} }}
		\quad
		\subfloat[\centering $\psi_{\theta,JDA}^N$ plot for $n=1000$, $0\leq\theta\leq1$ as $\{N\}_{500}^{1500}$ increases (reflected jump diffusion)]{{\includegraphics[scale=0.4]{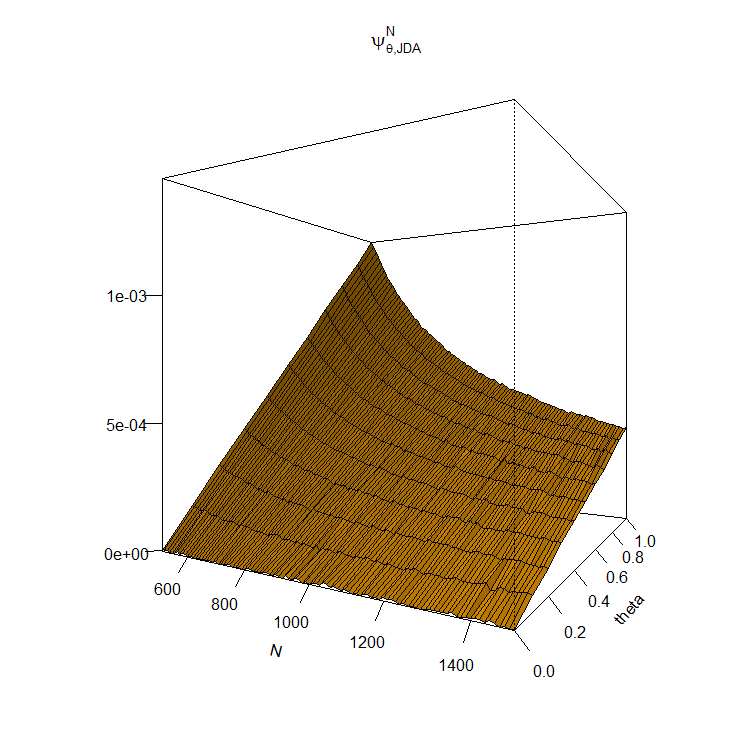} }}
		\quad
		\subfloat[\centering $\psi_{\theta,JMP}$ and $\psi_{\theta,JDA}^N$ plots for $n=N=100$, $0\leq\theta\leq100$.]{{\includegraphics[scale=0.3]{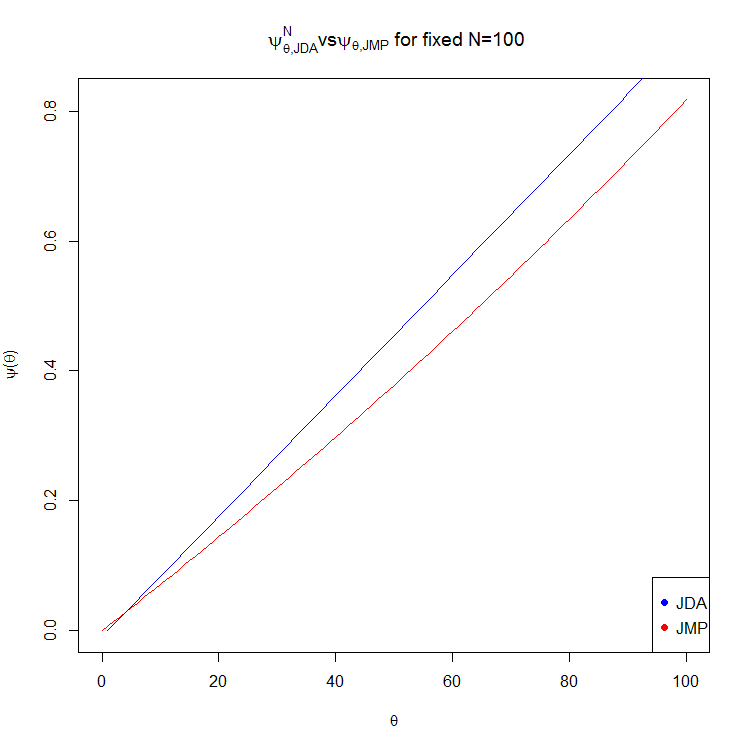} }}
		\quad
		\subfloat[\centering $\psi_{\theta,JMP}$ and $\psi_{\theta,JDA}^N$ plots for $n=N=1000$, $0\leq\theta\leq100$.]{{\includegraphics[scale=0.3]{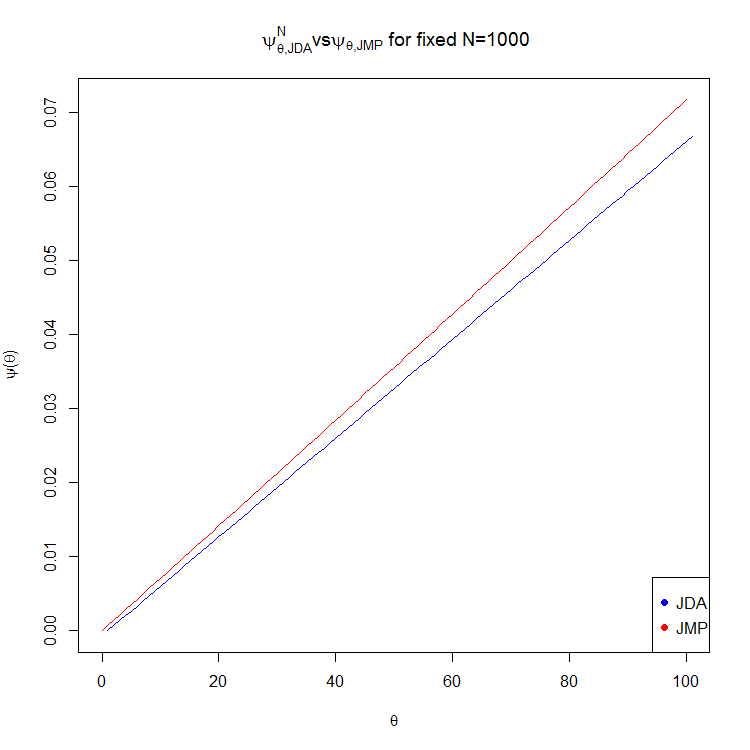} }}
		\caption{Numerical estimates of $\psi_{\theta}$ for the long-term local time of $\{0.25,0.75\}$ of the jump Markov process $\psi_{\theta,JMP}$ and of the reflected jump-diffusion $\psi_{\theta,JDA}^N$ (case $\gamma_n=10n^2$).}
		\label{group2}
	\end{adjustwidth}
\end{figure}  

\noindent Derivatives of $\psi_{\theta}$ near $\theta=0$ inferred from our results for $n=N=1000$ using centered finite differences at points $\theta=\{-0.01,0,0.01\}$:
\begin{align*}
&\psi_{0,JMP}' \approx 7.1\times10^{-3},\quad  \psi_{0,JMP}'' \approx 2.9\times10^{-5};\\
&\psi_{0,JDA}' \approx 9.1\times10^{-3},\quad  \psi_{0,JDA}''\approx 6.9\times10^{-2}
\end{align*}
indicate that in this scenario the reflected jump-diffusion is a closer approximation of the long term behaviour of the original model. This is a consequence of the fact that the dominant noise comes from jumps that are now present in the same (un-approximated) form in the reflected jump-diffusion as specified in the original model.  

Since the magnitude of division error rates is not influenced by the rates of chemical reactions in the system, it is crucial to have a way to distinguish their order of magnitude, which we argue can be done using time-additive functionals (dynamical observables) within experiments. One possible indicator is the local time at the deterministic stable equilibria: when $\gamma_n\gg n$ (e.g.\,$\gamma_n=10n^2$) both the jump Markov model and the reflected jump-diffusion from \eqref{eq-Vcrn-clejump} spend less time near equilibria, as the stronger noise from division errors counteracts the pull towards the stable equilibria of the drift $\mu$ which is defined by the system of reactions (\cite{McSP14}). Since we have numerically obtained the limiting cumulant generating function for the model with $\gamma_n=n$, we can argue that the comparison of the average value of this local time with that of its average value for the model with $\gamma_n=10n^2$ provides a distinction between the two cases.

To establish a quantifiable indicator of distinction between the two cases we also measure the long term {local times} of the jump Markov model and of the reflected jump-diffusion \eqref{eq-Vcrn-clejump}   in a neighborhood of the boundary $B_{1/(N+1)}(\{0,1\})$. All the in the numerical algorithm parameters are the same as above, except that the piecewise linear continuous functions $f$ approximating $\textbf{1}_{x\in B_{1/(N+1)}(0)\cup B_{1/(N+1)}(1)}$ are now:
\begin{align}\label{eq-f-bdry}
f(x)=\begin{cases}
1 & 0\leq x\leq \frac{1}{N+1} \\
2-(N+1)x & \frac{1}{N+1}\leq x < \frac{2}{N+1}\\
0 & \frac{2}{N+1} \leq x < 1-\frac{2}{N+1}\\
2+(N+1)(x-1) & 1-\frac{2}{N+1} \leq x < 1-\frac{1}{N+1}\\
1 & 1-\frac{1}{N+1} \leq x \leq 1.
\end{cases}\nonumber
\end{align}
An extra mesh point was used to account for the asymmetry of the neighborhood around the boundary points, which results in more numerically stable outputs. 

Figure \ref{group3} displays the results, and the derivatives near $\theta=0$ for this $\psi_{\theta}$ using $n=N=1000$ and centered finite differences at $\theta=\{-0.01,0,0.01\}$ are approximately:
\begin{align*}
&\psi_{0,JMP}' \approx 2.6\times10^{-1},\quad  \psi_{0,JMP}'' \approx 6.6\times10^{-3};\\
&\psi_{0,JDA}' \approx 1.8\times10^{-1},\quad  \psi_{0,JDA}''\approx 5.2\times10^{-2}.
\end{align*}
This confirms that the jump-diffusion approximation also accurately reflects the fact that the original process spends substantially more time reflecting at the boundaries than staying near its stable equilibria. In fact, since the noise in the model (after rescaling time) is not small, as the rescaling parameter $n$ increases the paths of the process in finite time are not converging closer to a deterministic path. Instead they exhibit fast passages through the interior of $(0,1)$ spending most of the time waiting for a reaction on the boundary $\{0,1\}$ to push it away from the  boundary  back into the interior. This fully matches the observations of the functional path behaviour of this model in \cite{McSP14}. Thm Prop 4.1 and Figure 3 in Sec 4.2. This type of behaviour has also been called `discreteness-induced transitions' when analyzed in related models (c.f. \cite{TK01, BDMcK14, BKW20}.

A comparison of the average long term local time near the boundaries versus the average time at  the stable equilibria then presents a quantifiable indicator for distinguishing the order of magnitude of division error intensity, and for retaining jumps in the reflected jump-diffusion approximation of the original model as crucial in the second case. Our example illustrates that numerical estimates of {local times} constitute valuable dynamic observables for reflected processes with drift of this form. They can be used to assess the closeness of approximating processes to the original model in the long-term (infinite time), as well as to distinguish the order of magnitude of an unbiased source of noise unseen by the drift of the process.

\begin{figure}[h]
	\begin{adjustwidth}{-1in}{-1in}
		\centering
	
		\subfloat[\centering $\psi_{\theta,JDA}^N$ plot for $n=100$, $0\leq\theta\leq1$ as  $\{N\}_{50}^{150}$ increases]{{\includegraphics[scale=0.4]{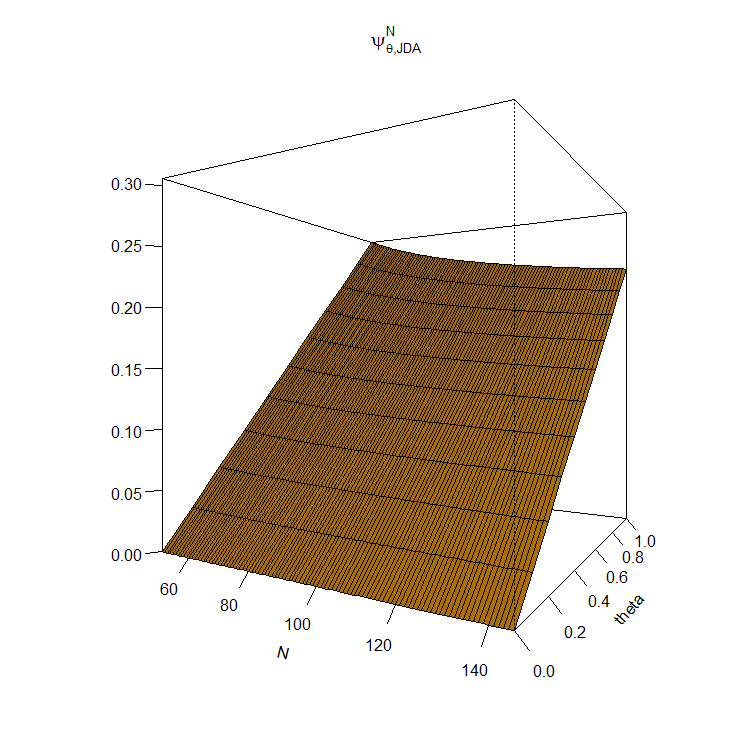} }}
		\quad
		\subfloat[\centering $\psi_{\theta,JDA}^N$ plot for $n=1000$, $0\leq\theta\leq1$ as $\{N\}_{500}^{1500}$ ]{{\includegraphics[scale=0.4]{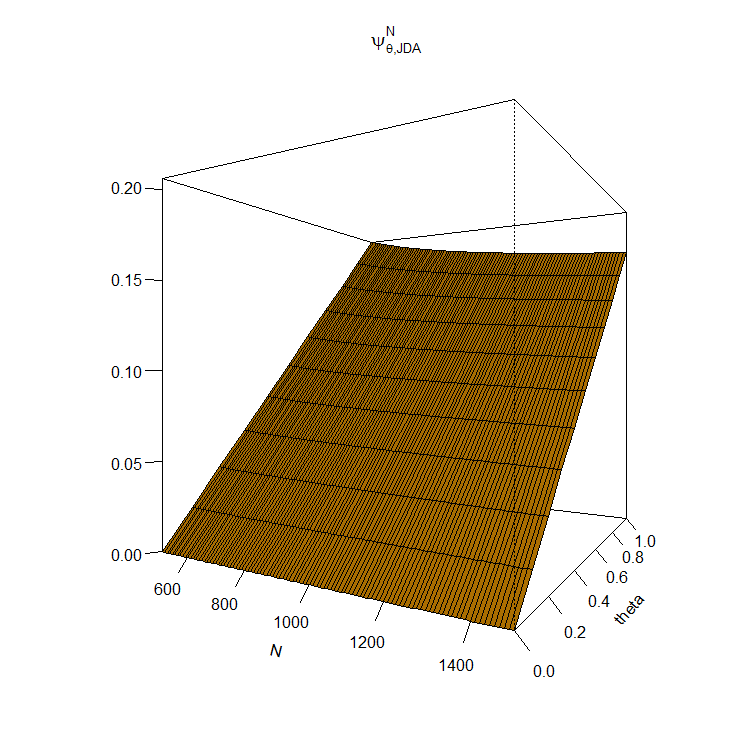} }}
		\quad
		\subfloat[\centering $\psi_{\theta,JMP}$ and $\psi_{\theta,JDA}^N$ plots for $n=N=100$, $0\leq\theta\leq100$.]{{\includegraphics[scale=0.3]{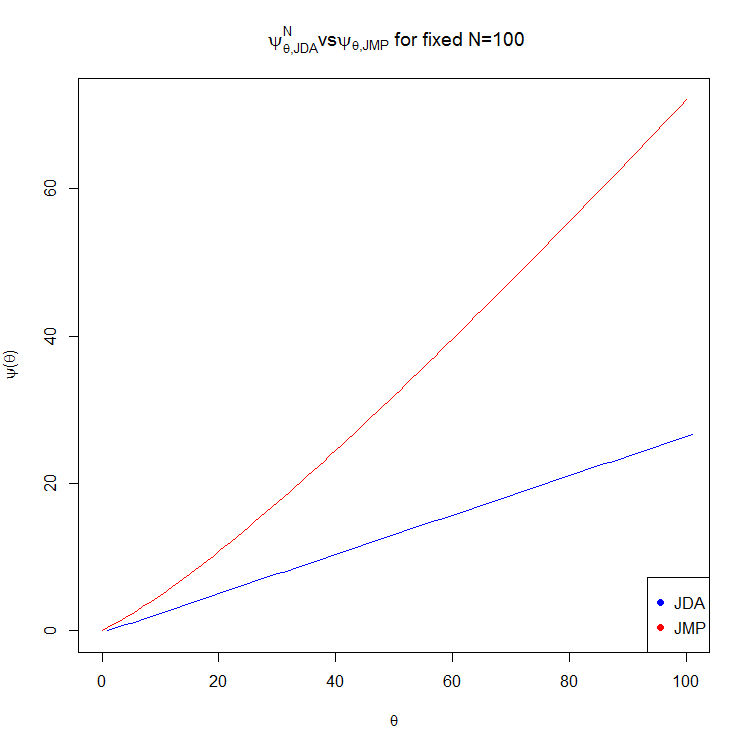} }}
		\quad
		\subfloat[\centering $\psi_{\theta,JMP}$ and $\psi_{\theta,JDA}^N$ plots for $n=N=1000$, $0\leq\theta\leq100$.]{{\includegraphics[scale=0.3]{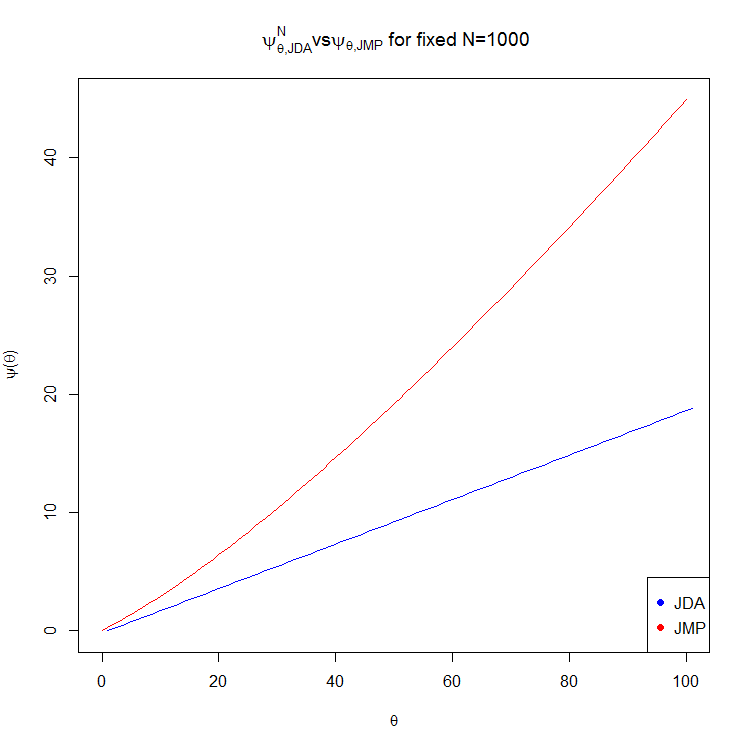} }}
		\caption{Numerical estimates of $\psi_{\theta}$ for long-term {local time} of $\{0,1\}$ of the jump Markov process $\psi_{\theta,JMP}$ and of the reflected jump-diffusion $\psi_{\theta,JDA}^N$ ($\gamma_n=10n^2$)}		
		\label{group3}
	\end{adjustwidth}
\end{figure}

A last comment on the stability of our numerical scheme in this example: we found that mesh sizes comparable to jump sizes in the process work well, while further refinements can lead to numerical instabilities.


%

\vspace{0.5cm}

{\bf Acknowledgments}. This work was supported by the Natural Sciences and Engineering Council of Canada (NSERC) Discovery Grant of the first author, and most of the research was conducted as part of the Masters thesis of the second author.  The authors would like to thank the anonymous readers whose questions have greatly improved the foundational basis for the paper and emphasized its relevance to our motivating application. 
\vspace{0.5cm}

\section{Appendix: Numerical method for the limiting logarithmic moment generating function}

Fix $\theta\in\mathbb{R}$. Let  $\tilde{\mathcal{L}}$ be the operator 
\[
\tilde{\mathcal{L}}u_\theta(x) = \mu(x)\partial_xu_\theta(x)+\frac{\sigma^2(x)}{2}\partial^2_{xx}u_\theta(x)+\theta f(x)u_{\theta}(x)+\int_{\mathcal{M}}[u_{\theta}(r(x,y))-u_{\theta}(x)]\nu_x(dy).
\]
defined on 
\[D(\tilde{\mathcal{L}})=\{u_\theta\in C^2([0,b]): \theta f(0) u_{\theta}(0)+ \rho_0\partial_xu_\theta(0)=0,\; \theta f(b) u_{\theta}(b)-\rho_b\partial_xu\theta(b)=0\}\]
We will numerically solve for the eigenvalue problem $\tilde{\mathcal{L}}u_\theta(x) =\psi_{\theta}u_{\theta}(x)$, $x\in [0,b]$ subject to the boundary conditions of $D(\tilde{\mathcal{L}})$, by replacing each derivative by an appropriate finite-difference quotient and the integral term by an appropriate sum.

Select an integer $N>0$ and divide the length of $[0,b]$ into $(N+1)$ equal subintervals whose endpoints are the mesh points $x_i=ih$, for $i=0,1,...,N+1$, where $h=\frac{b}{N+1}$. At the interior mesh points, $x_i$, for $i=1,2,...,N-1,N$, the PIDE to be approximated is $\tilde{\mathcal{L}}u_\theta(x_i) =\psi_{\theta}u_{\theta}(x_i)$.

Since $\nu_x(dy)$ may be discrete, continuous, or some combination of both, $\mathcal{M}$ as a series of disjoint continuous and discrete intervals. We call the continuous intervals $\mathcal{M}_{j_c}^c$ for $j_c=1,2,...,n_c$ and the discrete intervals $\mathcal{M}_{d_c}^d$ for $j_d=1,2,...,n_d$. Then we may rewrite the integral term as 
\begin{align*}
\int_{\mathcal{M}}[u_{\theta}(r(x,y))-u_{\theta}(x)]\nu_x(dy) =& \sum_{j_c}\int_{\mathcal{M}_{j_c}^c}[u_{\theta}(r(x,y))-u_{\theta}(x)]\nu_x(dy)
+ \sum_{j_d}\int_{\mathcal{M}_{j_d}^d}[u_{\theta}(r(x,y))-u_{\theta}(x)]\nu_x(dy). 
\end{align*} 
The discrete measure may be interpreted as
\begin{align*}
\sum_{j_d}\int_{\mathcal{M}_{j_d}^d}[u_{\theta}(r(x,y))-u_{\theta}(x)]\nu_x(dy) = \sum_{j_d}\sum_{y\in\mathcal{M}_{j_d}^d}[u_{\theta}(r(x,y))-u_{\theta}(x)]\nu_x(y).
\end{align*}
To approximate the continuous integral, we apply the Composite Trapezoidal rule over each interval $\mathcal{M}_{j_c}^c\equiv[a_{j_c},b_{j_c}]$. Define the integrand as $g^x(y)\equiv[u_{\theta}(r(x,y))-u_{\theta}(x)]\nu_x(y)$. For each $j_c$, select an integer $N_{j_c}>0$ and divide the length of $\mathcal{M}_{j_c}^c$, $(b_{j_c}-a_{j_c})$, into $N_{j_c}$ subintervals. So we have step size $h_{j_c}=\frac{b_{j_c}-a_{j_c}}{N_{j_c}}$ and $y_k^{j_c}=a_{j_c}+kh_{j_c}$ for each $k=0,1,...,N_{j_c}$. Then the first sum can be written as
\begin{align*}
\sum_{j_c}\int_{\mathcal{M}_{j_c}^c}g^x(y)dy =& \sum_{j_c}\left(\frac{h_{j_c}}{2}\left(g^x(a_{j_c}) + 2\sum_{j_c=1}^{N_{j_c}-1}g^x(y_k^{j_c})+g^x(b_{j_c})\right)-\frac{b_{j_c}-a_{j_c}}{12}h_{j_c}^2(g^x)''(\kappa_{j_c})\right)
\end{align*} 
for some $\kappa_{j_c}$ in $[a_{j_c},b_{j_c}]$.

We approximate the derivatives on $(0,b)$ by a centered-difference scheme where $\eta_i$ and $\xi_i$ are some values in $(x_{i-1},x_{i+1})$. Putting all the approximations together results in the finite difference equation
\begin{align*}
&\mu(x_i)\left(\frac{u(x_{i+1}) - u(x_{i-1})}{2h} - \frac{h^2}{6}u^{(3)}(\eta_i) \right) +\frac{\sigma^2(x_i)}{2}\left( \frac{u(x_{i+1}) -2u(x_i) + u(x_{i-1})}{h^2} - \frac{h^2}{12}u^{(4)}(\xi_i) \right) \\
&+\theta f(x_i)u(x_i)+ \sum_{j_c}\left(\frac{h_{j_c}}{2}\left(g^{x_i}(a_{j_c}) + 2\sum_{j_c=1}^{N_{j_c}-1}g^{x_i}(y_k^{j_c})+g^{x_i}(b_{j_c})\right)-\frac{b_{j_c}-a_{j_c}}{12}h_{j_c}^2(g^{x_i})''(\kappa_{j_c})\right)\\ 
&+\sum_{j_d}\sum_{y\in\mathcal{M}_{j_d}^d}[u(r(x_i,y))-u(x_i)]\nu_{x_i}(y) = \psi_{\theta}u(x_i).
\end{align*}

We choose forward and backward finite-difference schemes with $O(h^2)$ truncation error to approximate the boundary conditions:
\[
\frac{-\frac{3}{2}u_0 + 2u_1 -\frac{1}{2}u_2}{h} = -\frac{f(0)\theta u_0}{\rho_0} \iff  u_0 = \frac{4\rho_0 u_1-\rho_0 u_2}{3\rho_0-2\theta f(0) h},
\]
\[
\frac{\frac{3}{2}u_{N+1}-2u_N+\frac{1}{2}u_{N-1}}{h} = \frac{f(b)\theta u_{N+1}}{\rho_b} \iff u_{N+1} = \frac{\rho_b u_{N-1}-4\rho_b u_N}{2\theta f(b) h - 3\rho_b}.
\]

After truncating and rearranging together with the boundary conditions, we define the system of linear equations
\begin{align*}
&\left( -\frac{\mu(x_i)}{2h} +\frac{\sigma^2(x_i)}{2h^2} \right)u_{i-1} + \left(-\frac{\sigma^2(x_i)}{h^2} + \theta f(x_i) \right)u_i + \left( \frac{\mu(x_{i})}{2h} + \frac{\sigma^2(x_i)}{2h^2} \right)u_{i+1}+ \sum_{l=0}^{N+1} \tilde{g}^{u_{l}}(x_i)u_l=\psi_{\theta}u_i,
\end{align*}
where each function $\tilde{g}^{u_{l}}(x_i)$ represents the sum of all terms in the integral approximations that are factors of $u_l$, for each $i=1,2,...,N$. Every time $u_0$ or $u_{N+1}$ is a term in the equation, we replace it with the appropriate boundary condition. This allows us to define the system of equations as an $N\times N$ matrix, with $a_1(x_i) \equiv  -\frac{\mu(x_i)}{2h} +\frac{\sigma^2(x_i)}{2h^2} $, $a_2(x_i)\equiv-\frac{\sigma^2(x_i)}{h^2} + \theta f(x_i)$, and $a_3(x_i)\equiv \frac{\mu(x_{i})}{2h} + \frac{\sigma^2(x_i)}{2h^2}$, as
\[
(\textbf{A}+\textbf{G})\textbf{u}=\psi_{\theta}\textbf{u},
\]
where
\[	
\textbf{A} = 
\begin{bmatrix}
\frac{4\rho_0a_1(x_1)}{3\rho_0-2\theta f(0) h} + a_2(x_1)   &  -\frac{\rho_0 a_1(x_1)}{3\rho_0-2\theta f(0) h}+a_3(x_1)  & 0 & \cdots & 0 & 0\\
a_1(x_2)  & a_2(x_2)  &  a_3(x_2)  & \cdots & 0 & 0\\
\ddots & \ddots & \ddots & \ddots & \ddots & \ddots\\
0 & \cdots & 0 & a_1(x_{N-1}) & a_2(x_{N-1})  &  a_3(x_{N-1})\\
0 & \cdots & 0 & 0 & a_1(x_N) + \frac{\rho_b a_3(x_N)}{2\theta f(b) h - 3\rho_b} & a_2(x_N) - \frac{4\rho_b a_3(x_N)}{2\theta f(b) h - 3\rho_b} \\
\end{bmatrix}
\text{,}
\]
\[\textbf{G} = 
\begin{bmatrix}
\tilde{g}^{u_{1}}(x_1) & \tilde{g}^{u_{2}}(x_1) & \tilde{g}^{u_{3}}(x_1) &\cdots & \tilde{g}^{u_{N}}(x_1)\\
\tilde{g}^{u_{1}}(x_2) & \tilde{g}^{u_{2}}(x_2) & \tilde{g}^{u_{3}}(x_2) & \cdots &\tilde{g}^{u_{N}}(x_2)\\
\ddots & \ddots & \ddots & \ddots & \ddots\\
\tilde{g}^{u_{1}}(x_N) & \tilde{g}^{u_{2}}(x_N) & \tilde{g}^{u_{3}}(x_N) & \cdots &\tilde{g}^{u_{N}}(x_N)\\
\end{bmatrix}
\text{, and }
\textbf{u} = 
\begin{bmatrix}
u_1\\u_2\\\vdots\\u_N
\end{bmatrix}
\text{.}
\]
Finally, we compute $\psi_{\theta}$ by solving for the eigenvalues of \textbf{A}+\textbf{G} and selecting the one with largest value, then its associated eigenvector is the solution $\textbf{u}$ to the PIDE.\\

\end{document}